\documentclass{amsart}

\usepackage{amssymb,amsmath}
\usepackage{tikz-cd}
\usepackage{hyperref}
\usepackage{cases}
\usepackage{cleveref}

\newtheorem{theorem}{Theorem}[section]
\newtheorem{prop}[theorem]{Proposition}
\newtheorem{lemma}[theorem]{Lemma}
\newtheorem{cor}[theorem]{Corollary}

\theoremstyle{definition}

\theoremstyle{remark}
\newtheorem{remark}[theorem]{Remark}

\DeclareMathOperator\tr{tr}
\def\ep{\varepsilon}    
\newcommand{\R}{\mathbb{R}}

\begin{document}


\title{A new conformal heat flow of harmonic maps}

\author{Woongbae Park}
\address{522 Thackeray Hall, Department of Mathematics, University of Pittsburgh, Pittsburgh, Pennsylvania, 15260}
\email{wop5@pitt.edu}



\begin{abstract}
We introduce and study a conformal heat flow of harmonic maps defined by an evolution equation for a pair consisting of a map and a conformal factor of metric on the two-dimensional domain.
This flow is designed to postpone finite time singularity but does not get rid of possibility of bubble forming.
We show that Struwe type global weak solution exists, which is smooth except at most finitely many points.
\end{abstract}

\maketitle
\sloppy

\section{Introduction}
\label{sec1}

Consider a map $f_0 : \Sigma \times [0,T) \rightarrow N$ from a compact Riemann surface $(\Sigma,g_0)$ with metric $g_0$ to a Riemannian manifold $(N,h)$.
Under the usual harmonic map heat flow, $f_0$ evolves to a map $f(t)$ according to the evolution equation $f_t = \tau_{g_0}(f)$, where $\tau_g(f) = \tr_g (\nabla^g df)$ is the tension field with respect to the metric $g$.
In this paper we consider the generalization in which both the map and the metric evolve with $(f(t),g(t))$ satisfying the equations
\begin{subnumcases}{\label{main}}
f_t \ = \tau_g(f) \label{main1}\\
g_t \ = (2b|df|_g^2 -2a)g \label{main2}
\end{subnumcases}
where $a,b>0$ are constants and $|df|_g^2 = g^{ij}h_{\alpha \beta} f_i^{\alpha}f_j^{\beta}$ is the energy density.
We assume that the initial map $f(0)=f_0$ and metric $g(0)=g_0$ are smooth.

The first of these equations is the harmonic map heat flow, with varying metric $g$.
The second equation is designed to attenuate energy concentration.
If the energy density become large in some region $\Omega \subset \Sigma$, then under the flow \eqref{main2}, the metric is conformally enlarged; this increases the area of $\Omega$ and decreases the energy density.
This suggests that the system \eqref{main} may be better behaved than the harmonic map heat flow, where energy concentration at points is an impediment to convergence.

Writing the metric $g(t) = e^{2u} g_0$ for a real-valued function $u(t)$, equations \eqref{main} are equivalent to the following equations for the pair $(f(t),u(t))$:
\begin{subnumcases} {\label{main0}}
f_t \ = e^{-2u}\tau(f) \label{main0-1}\\
u_t \ = b e^{-2u}|df|^2 - a \label{main0-2}
\end{subnumcases}
where $\tau$ and $|df|^2$ are with respect to the fixed metric $g_0$, and where the initial conditions are $f(0)=f_0$, $u(0)=0$.
In this form, the flow is more easily analyzed.

The main Theorem of this paper is the following.

\begin{theorem}\label{global solution}
(Existence of global weak solution)
For any $f_0 \in W^{3,2}(\Sigma,N)$, a global weak solution $(f,u)$ of \eqref{main0} exists on $\Sigma \times [0,\infty)$ which is smooth on $\Sigma \times (0,\infty)$ except at most finitely many points.
\end{theorem}

There is a long history of harmonic maps and related fields.
We could not list all such literatures but only few, including \cite{EL78}, \cite{EL88}, \cite{SU81}, \cite{S84}, \cite{J08}, \cite{H91}, \cite{P96}, \cite{LW98}, \cite{CT99}, \cite{R07}, \cite{Z10}, \cite{CLW12} and therein.
In terms of heat flow of harmonic maps, see for example \cite{ES64}, \cite{S85}, \cite{CG89}, \cite{CD90}, \cite{S90}, \cite{F95}, \cite{F95b}, \cite{T97}, \cite{T02}, \cite{T04}, \cite{T04b}, \cite{LW08}, \cite{HW16} and therein.
Note that usual heat flow can have finite time singularity, see Chang-Ding-Ye \cite{CDY92}, Raphael-Schweyer \cite{RS14}, or more recently D{\'a}vila-Del Pino-Wei \cite{DDW20}.

There are several directions to allow metric change along harmonic map heat flow.
The most well-known direction is Teichm{\"u}ller flow, where metric lies in Teichm{\"u}ller space of constant curvature.
Teichm{\"u}ller flow is the $L^2$ gradient flow of the energy and hence reduce the energy in the fastest sense.
For relevant literature, see for example Rupflin-Topping \cite{RT18a}, Huxol-Rupflin-Topping \cite{HRT16} or Rupflin-Topping \cite{RT19} and therein.
Another direction is Ricci-harmonic map flow.
This is a combination of harmonic map heat flow and Ricci flow of the metric.
Surprisingly, this flow is more regular than both harmonic map heat flow and Ricci flow.
See for example, Muller \cite{M12}, Williams \cite{W15} or Buzano-Rupflin \cite{BR17} among others.
Recently in Huang-Tam \cite{HT21}, harmonic map heat flow together with evolution equation of metric is considered under time-dependent curvature restriction and smooth short time existence is obtained.
Because we do not assume a priori curvature bounds of the domain, the result cannot be applied into our case.

The paper is organized as follows.
In Section \ref{sec2} we look at some preliminaries, including volume formula and its asymptotic limit if the map $f$ is steady solution, that is, harmonic.
Next, in Section \ref{sec3} we define Hilbert spaces $X,Y,Z$ and their closed subsets $B,B'$.
So, from Section \ref{sec3} we consider $f \in B$ and $u \in B'$.
Then Section \ref{sec4} defines the operator $S_1, S_2$ and shows their properties.
Briefly, we can show that $S_1 : B \times B' \to B$ and $S_2 : B \times B' \to B'$ and they satisfy twisted partial contraction properties, see Lemmas \ref{P1Lemma2}, \ref{P1Lemma3}, \ref{P2Lemma2}, and \ref{P2Lemma3}.
In the last section \ref{sec5} we define the operator $\mathcal{S}$ on $B \times B'$ mapping into itself defined by $\mathcal{S} = (S_1,S_2)$.
For $T$ small enough, $\mathcal{S}$ is a contraction and hence we can prove short time existence.

Next we are working on types of singularity.
Ultimately we will show that the solution is singular only when energy concentrates, similar with Struwe's result.
In Section \ref{sec6} we show local estimate and obtain bounds for $\iint e^{2u}|f_t|^4$.
This is used in Section \ref{sec7} to show $W^{2,2}$ and higher estimate, which implies boundedness of $|df|$.
Finally in Section \ref{sec8} we prove the main theorem \ref{global solution}.

\subsection{Notation}

Even though our equation is heat-type equation for varying metric, we use initial metric $g_0$ as default.
So, all terms using metric use $g_0$ unless we specify the metric.
For example, $|df|^2$ is calculated in terms of $g_0$ and $|df|_g^2$ is calculated in terms of $g$.
If the volume form is calculated in terms of metric $g$, we denote it as $dvol_g$.
We also omit  $dvol_{g_0}$ and $dt$ if there is no confusion.
We also use the simplifications $\|\cdot\|_{W^{k,p}} = \|\cdot\|_{W^{k,p}(\Sigma \times [0,T])}$, $\|\cdot\|_{C^0} = \|\cdot\|_{C^0(\Sigma \times [0,T])}$ and $\|\cdot\|_{L^p} = \|\cdot\|_{L^p(\Sigma \times [0,T])}$.
Also, the constant $c$ is universal and changed line by line.

\section{Preliminaries}
\label{sec2}

Before we show the main result, we record a few facts about solutions to the flow equations \eqref{main0}.

\subsection{Energy and Volume}

First note that the 2-form $|df|^2 \, dvol_g$ is conformally invariant, and that the energy
\begin{equation} \label{2.energy}
E(t) = \tfrac{1}{2} \int |df|^2 \, dvol_g
\end{equation}
satisfies
\begin{equation}\label{E'<0}
E'(t) =  \int \ \langle df, df_t\rangle\ =\ - \int \langle \nabla df, e^{-2u}\tau(f)\rangle\ =\  - \int   e^{-2u} |\tau(f)|^2\ \le\ 0.
\end{equation}
Thus $E(t)\le E_0$ for all $t$.\\

\begin{lemma}
The volume satisfies $V(t)\le e^{-2at} V(0) + \tfrac{2b}{a} E_0$, and hence is finite for all $t$.
\end{lemma}

\begin{proof}
The second equation \eqref{main0-2} can be explicitly solved, yielding
\begin{equation}\label{eq:2.1}
e^{2u} = e^{-2at} \left( 1 + 2b \int_{0}^{t}e^{2as}|df|^2 (s)ds \right). 
\end{equation}
Consequently,   the volume
\begin{equation}\label{eq:2.2}
V(t) = \int_{\Sigma} dvol_{g(t)} = \int_{\Sigma} e^{2u} dvol_{g_0}
\end{equation}
can be written as
\begin{equation}\label{eq:2.3}
V(t) = e^{-2at} \left( V(0) + 4b \int_{0}^{t}e^{2as} E(s)ds \right).
\end{equation}
The lemma follows by noting that $E(s)\le E_0$ and integrating.
\end{proof}

\subsection{Asymptotic behavior of steady solution}

Now we consider steady solution.

\begin{lemma} \label{lemma 2.3}
Let $(f,u)$ be a solution of \eqref{main0} and $f(0)$  a harmonic with energy $E$.
Then $f(t)$ is harmonic for all $t$ and as $t \rightarrow \infty$,
\begin{equation*}
e^{2u} \rightarrow \tfrac{b}{a}|df|^2
\end{equation*}
and hence by \eqref{eq:2.2} the volume $V(t)$ converges to
\begin{equation*}
V(\infty) = \tfrac{2b}{a} E.
\end{equation*}
\end{lemma}

\begin{proof}
If $f(0)$ is harmonic, then $f_t=0$ and hence $f$ and $|df|^2$ are independent of $t$.
Integrating \eqref{eq:2.1} then shows that, as $t \rightarrow \infty$,
\begin{align*}
e^{2u} &= e^{-2at} \left( 1 + 2b |df|^2 \frac{e^{2at}-1}{2a} \right)\\
&= e^{-2at} + \tfrac{b}{a} |df|^2 (1-e^{-2at}) \rightarrow \tfrac{b}{a} |df|^2.
\end{align*}
\end{proof}

This means that, for solutions as in Lemma \ref{lemma 2.3}, the energy density $|df|_g^2 = |df|^2 e^{-2u}$ converges as $t \to \infty$ to the constant $\frac{a}{b}$.
Hence the conformal heat flow forces the conformal factor and the energy density be distributed evenly.
Remark that, because the image $f(\Sigma)$ does not change, this flow modifies the domain toward the space which is similar to the image with the similarity ratio $\frac{a}{b}$.\\

\section{Construction of Hilbert spaces}
\label{sec3}

In this section we build Hilbert spaces $X_T,Y_T,Z_T$ and their closed subsets $B,B'$.
For parabolic theory used here, see Mantegazza-Martinazzi \cite{MM12}, Evans \cite{E10} or Lieberman \cite{L96}.
From now on, we consider the target manifold being isometrically embedded, $N \hookrightarrow \mathbb{R}^L$.

\subsection{Spaces $X$, $Y$ and $Z$}

The set
\[
Y_T\ = L^2([0,T],W^{4,2}(\Sigma,\mathbb{R}^L)) \cap W^{1,2}([0,T],W^{2,2}(\Sigma,\mathbb{R}^L)) \cap W^{2,2}([0,T],L^2(\Sigma,\mathbb{R}^L))
\]
is a Hilbert space with norm
\[
\|f\|^2_Y\ =\ \int_0^T\int_\Sigma\  |\nabla^4f|^2 + |f|^2+|\nabla^2 f_t|^2 + |f_t|^2+|f_{tt}|^2 \ dvol_{g_0} \ dt.
\]


As in Proposition 4.1 in \cite{MM12},
\[
Y_T \hookrightarrow C^0([0,T],C^1(\Sigma,\mathbb{R}^L)) \cap L^4([0,T],W^{3,4}(\Sigma,\mathbb{R}^L)) \cap W^{1,4}([0,T],W^{1,4}(\Sigma,\mathbb{R}^L))
\]
and there is a constant ${c}$ such that 
\begin{align}\label{A1}
\|f\|_{C^0} + \|\nabla f\|_{C^0} + \| \nabla^3 f\|_{L^4} + \|\nabla f_t\|_{L^4} \le {c}\|f\|_Y.
\end{align}
Also, by standard parabolic theory (See, for example, \cite{E10}), $f \in Y_T$ implies $f \in C^0([0,T],W^{3,2}(\Sigma,\mathbb{R}^L))$, $f_t \in C^0([0,T],W^{1,2}(\Sigma,\mathbb{R}^L))$ and
\begin{align}\label{A1'}
 \max_{0 \le t \le T} \| f(t)\|_{W^{3,2}(\Sigma)}, \max_{0 \le t \le T}  \| f_t(t)\|_{W^{1,2}(\Sigma)} \le {c} \|f\|_Y.
\end{align}

This also implies that
\begin{align} \label{A1''}
\max_{0 \le t \le T} \|f(t)\|_{W^{2,8}(\Sigma)} \le {c} \|f\|_Y.
\end{align}

Next, denote
\[
X_T = L^2([0,T],W^{2,2}(\Sigma,\mathbb{R}^L)) \cap W^{1,2}([0,T],L^2(\Sigma,\mathbb{R}^L))
\]
be another Hilbert space with norm
\[
\|f\|^2_{X} = \int_{0}^{T} \int_{\Sigma} |f|^2 + |\nabla^2 f|^2 + |f_t|^2 \ dvol_{g_0} \ dt.
\]

Note that in the notation of \cite{MM12}, $Y = P^2$ and $X = P^1$.


Now we define spaces for $u$.
The set
\[
Z_T\ = L^2([0,T],W^{3,2}(\Sigma)) \cap W^{1,2}([0,T],W^{1,2}(\Sigma))
\]
is a Hilbert space with norm
\[
\|u\|_{Z}^2 \ = \  \int_0^T\int_\Sigma\  |\nabla^3 u|^2 + |u|^2+|\nabla u_t|^2 + |u_t|^2 \ dvol_{g_0} \ dt.
\]
Similar to above, there is a constant ${c}$ such that
\begin{equation} \label{AZ1}
\|\nabla^2 u\|_{L^4} + \|u_t\|_{L^4} \le {c} \|u\|_Z
\end{equation}
and
\begin{equation} \label{AZ2}
\max_{0 \le t \le T}  \|u(t)\|_{W^{2,2}(\Sigma)} + \max_{0 \le t \le T}  \|u_t(t)\|_{L^2(\Sigma)} \le {c} \|u\|_{Z}.
\end{equation}
Also, by Sobolev embedding, we have
\begin{equation} \label{AZ3}
\max_{0 \le t \le T} \|u(t)\|_{W^{1,8}(\Sigma)} \le {c} \|u\|_Z.
\end{equation}
Moreover, $u$ is continuous and there is a constant $C_2$ such that for all $u \in Z_T$,
\begin{equation}
\|u\|_{C^0} \le C_2\|u\|_Z.
\end{equation}


\subsection{The ball $B$ and $B'$} 

Now we fix $f_0 \in W^{3,2}(\Sigma)$ throughout the section and thereafter.
Consider the operator $\partial_t - e^{-2u}\Delta$.
If $\|u\|_{C^0} \le 1$, this operator is uniformly elliptic.
So, Proposition 2.3 of \cite{MM12} then says that the map $ f\mapsto \big(f_0, (\partial_t-e^{-2u}\Delta)f \big)$ is a linear isomorphism
\[
Y_T \to W^{3,2}(\Sigma) \times X_T.
\]
Hence there is a constant $C_1$ such that for each $f_0 \in W^{3,2}(\Sigma)$ and $g \in X_T$, there is  a unique solution $h(t,x) \in Y_T$ of the initial value problem 
\begin{align}\label{A2}
(\partial_t-e^{-2u}\Delta)h=g \qquad h(0)=f_0
\end{align}
with
\begin{align}\label{A2b}
\|h\|_Y\le C_1 \big(\|f_0\|_{3,2} + \|g\|_{X}\big).
\end{align}

  Let $h_0(t,x)$ be the unique solution of 
  \begin{align}\label{Ah_0eq}
(\partial_t-\Delta)h=0 \qquad h(0)=f_0.
\end{align}
   By \eqref{A2b}  there is a constant $C_0$, depending on $C_1$ and $\|f_0\|_{3,2}$ such that
\begin{equation} \label{C0}
 \|h_0\|_{Y}\ \le\ C_0.
 \end{equation}

  Because of \eqref{A1},  $\big\{f\in Y_T\, \big|\, f(0)=f_0\big\}$ is a closed affine subspace of $Y_T$.  Hence the ball 
\begin{align}\label{Aball}
 B = B_\delta\  =\ \big\{f\in Y_{T}\, \big|\, f(0)=f_0 \ \mbox{and}\  \|f-h_0\|_Y\le \delta \big\}
\end{align}
is a closed subset of $Y_T$.  Note that each $f\in B_\delta$ satisfies
\begin{align}\label{Aballbound}
\|f\|_Y\ \le\ \|f-h_0\|_Y+\|h_0\|_Y\ \le\ \delta+C_0 .
\end{align}
Also let the ball
\[
B' = B_{\delta'}' = \{u \in Z_T \, \big| \, u(0) = 0 \ \mbox{and} \ \|u\|_{Z} \le \delta'\}
\]
be a closed subset of $Z_T$.
Obviously $h_0 \in B_\delta$ and $0 \in B_{\delta'}'$.
For simplicity, we denote $B= B_\delta$ and $B' = B_{\delta'}'$.


Now fix $\delta>0$ and define
\begin{equation} \label{C3}
C_3 := 1600 C_0 C_1 C_2.
\end{equation}
Choose $\delta'$ small enough so that $C_2 \delta' < 1$ which implies $\|u\|_{C^0} \le 1$.
Also we assume $\delta' \le \frac{\delta}{C_3}$.


\section{Construction of operators}
\label{sec4}

In this section we will construct operators $S_1 : Y_T \times Z_T \to Y_T$ and $S_2 : Y_T \times Z_T \to Z_T$.
First fix $f \in Y_T$ and $u \in Z_T$.
$f$ and $u$ are considered to be fixed throughout this section and after unless we mention any choice of them.

First we show a lemma that is needed in several places.


\begin{lemma}\label{P0Lemma}
Fix $f_0 \in W^{3,2}(\Sigma)$.
Then there is an $T_0 = T_0(C_0,\delta,\delta')>0$ such that for all $T \le T_0$, for each $h \in B$ and $u_1,u_2 \in B'$,
\begin{equation}\label{AT0}
\|(e^{2u_2-2u_1}-1) \partial_t h\|_X \le \frac{C_3}{2C_1}\|u_1-u_2\|_Z.
\end{equation}
\end{lemma}

\begin{proof}
Denote
\[
g := (e^{2u_2-2u_1}-1) \partial_t h.
\]
Recall that 
\[
\left|e^{2u_1-2u_2}-1 \right| \le e^{2|u_1-u_2|} \left|1-e^{-2|u_1-u_2|} \right| \le 2e^4 |u_1-u_2|
\]
if $\|u_1-u_2\|_{C^0} \le 2$, which comes from $u_1,u_2 \le B'$.
Using $2e^4 \le 200$ and by \eqref{A1'} and \eqref{Aballbound},
\begin{align*}
\|g\|_{L^2}^2 &\le 200^2 \|u_1-u_2\|_{C^0}^2 \max_{0 \le t \le T} \|\partial_t h_2(t)\|_{L^2}^2 \, T \\
&\le \frac{C_3^2}{16C_1^2}\|u_1-u_2\|_{Z}^2 
\end{align*}
if we choose $T$ small enough.

Next, consider $\left| \nabla^2 g \right|^2 $.
\begin{align*}
\left| \nabla^2 g \right| &= \left| \nabla^2 \left( (e^{2u_2-2u_1}-1) \partial_t h_2  \right) \right| \\
&\le 800 |u_1-u_2| |\nabla (u_1-u_2)|^2 |\partial_t h_2| + 400 |u_1-u_2| |\nabla^2 (u_1-u_2)| |\partial_t h_2|\\
&\quad + 400 |u_1-u_2| |\nabla (u_1-u_2)| |\nabla \partial_t h_2| + 200 |u_1-u_2| |\nabla^2 \partial_t h_2|.
\end{align*}
Hence, by integrating, we have
\begin{align*}
\|\nabla^2 g\|_{L^2}^2 &\le 1600^2 \|u_1-u_2\|_{C^0}^2 \|\nabla (u_1-u_2)\|_{L^8}^4 \max_{0 \le t \le T} \|\partial_t h_2(t)\|_{L^4(\Sigma)}^2 \, T^{1/2}\\
&\quad + 800^2 \|u_1-u_2\|_{C^0}^2 \|\nabla^2 (u_1-u_2)\|_{L^4}^2  \max_{0 \le t \le T} \|\partial_t h_2(t)\|_{L^8(\Sigma)}^4 \, T^{1/2}\\
&\quad + 800^2 \|u_1-u_2\|_{C^0}^2  \max_{0 \le t \le T} \|\nabla (u_1(t)-u_2(t))\|_{L^4(\Sigma)}^2 \|\nabla \partial_t h_2\|_{L^4}^2 \, T^{1/2}\\
&\quad + 400^2 \|u_1-u_2\|_{C^0}^2 \|\nabla^2 \partial_t h_2\|_{L^2}^2\\
&\le 400^2 C_2^2 \|u_1-u_2\|_Z^2 2 C_0^2\\
&= \frac{C_3^2}{8C_1^2} \|u_1-u_2\|_Z^2
\end{align*}
if we choose $T$ small enough.

Finally, we will compute  $\|g_t\|_{L^2}^2$.
\begin{align*}
|g_t| &\le 400 |u_1-u_2||\partial_t h_2| |(u_1-u_2)_t| + 200 |u_1-u_2| |\partial_{tt} h_2|.
\end{align*}
Hence,
\begin{align*}
\|g_t\|_{L^2}^2 &\le 2(400)^2 \|u_1-u_2\|_{C^0}^2 \|(u_1-u_2)_t\|_{L^4}^2 \max_{0 \le t \le T} \|\partial_t h_2\|_{L^4(\Sigma)}^2 \, T^{1/2}\\
&\quad + 2(200)^2 \|u_1-u_2\|_{C^0}^2 \|\partial_{tt} h_2\|_{L^2}^2  \\
&\le 2(200)^2 C_2^2 \|u_1-u_2\|_Z^2  2 C_0^2\\
&= \frac{C_3^2}{16C_1^2} \|u_1-u_2\|_Z^2
\end{align*}
if we choose $T$ small enough.

Combining all the estimates above, we get
\[
\|(e^{2u_2-2u_1}-1) \partial_t h\|_X \le \frac{C_3}{2C_1} \|u_1-u_2\|_Z
\]
which proves the lemma.
\end{proof}


\subsection{The construction $S_1$}  Define an operator
\[
S_1:  Y_T \times Z_T \to Y_T
\]
by $S_1(f,u)=h$ where $h\in Y_{T}$ is the unique solution of
\begin{align}\label{A3}
(\partial_t-e^{-2u}\Delta)h=e^{-2u} A_f(df,df)\qquad h(0)=f_0.
\end{align}


 \begin{lemma}\label{P1Lemma1}
 Fix  $f_0\in W^{3,2}(\Sigma)$.
  Then there is $T_0 = T_0(C_0,\delta,\delta') >0$ such that for all $T \le T_0$,
$S_1$ restricts to an operator $S_1 :  B \times B' \to  B$.
\end{lemma}

\begin{proof}
We also can assume $\|A\|, \|DA\|, \|D^2 A\|, \|D^3 A\| \le {c}$ where ${c}$ depends only on the geometry of $N$.
Then the  vector-valued function $A_f(df,df)$ satisfies the pointwise bound $|A_f(df,df)|^2\le {c}|df|^4$.
Fix $f \in B$ and $u \in B'$.

     Now we estimate $X$ norm of
 \[
 g=e^{-2u(t)} A_f(df,df).
 \]
 First, $|g|^2 \le {c}|df|^4$, so $\|g\|_{L^2}^2 \le {c}\|f\|_Y^4 |\Sigma| T$.
 Hence if we choose $T$ small enough, we have $\|g\|_{L^2}^2 \le \frac{\delta^2}{6C_1^2}$.
Next, compute $|\nabla^2 g|^2$.
 \begin{align*}
 |\nabla^2 g| &\le {c} |df|^2 |\nabla^2 u| + {c} |df|^2 |\nabla u|^2 + {c} |df|^3 |\nabla u| + {c}|\nabla df| |df| |\nabla u|\\
 &\quad + {c}|df|^4 + {c}|df|^2 |\nabla df| + {c} |\nabla^3 f| |df| + {c} |\nabla df|^2.
 \end{align*}
 So, using Young's inequality 
 \[
 {c}\|df\|_{C^0}^2 \iint |\nabla df|^2 |\nabla u|^2 \le {c}\|df\|_{C^0}^4 \iint |\nabla u|^4 + {c} \iint |\nabla df|^4,
 \]
 we get, by \eqref{A1}, \eqref{A1'}, \eqref{A1''}, \eqref{AZ2}  and \eqref{Aballbound},
 \begin{align*}
 \|\nabla^2 g\|_{L^2}^2 &\le {c} \|df\|_{C^0}^4 \iint (|\nabla^2 u|^2 + |\nabla u|^4) + {c} \|df\|_{C^0}^6 \iint |\nabla u|^2\\
 &\quad + {c} \|df\|_{C^0}^2 \iint |\nabla df|^2 |\nabla u|^2  + {c} \|df\|_{C^0}^8 |\Sigma| T + {c} \|df\|_{C^0}^4 \iint |\nabla df|^2 \\
 &\quad + {c} \|df\|_{C^0}^2 \iint |\nabla^3 f|^2 + {c} \iint |\nabla df|^4\\
 &\le {c} \|f\|_Y^4 \left( \max_{0 \le t \le T} \| \nabla^2 u(t)\|_{L^2(\Sigma)}^2 +  \max_{0 \le t \le T}  \|\nabla u(t)\|_{L^4(\Sigma)}^4 \right) T\\
 &\quad  + {c}\|f\|_Y^6  \max_{0 \le t \le T}  \|\nabla u(t)\|_{L^2(\Sigma)}^2 T + {c}\|f\|_Y^8 |\Sigma| T \\
 &\quad + {c} \|f\|_Y^4 \max_{0 \le t \le T} \|\nabla df (t)\|_{L^2(\Sigma)}^2 T + {c} \|f\|_Y^2  \max_{0 \le t \le T}\|\nabla^3 f(t)\|_{L^2(\Sigma)}^2 T \\
 &\quad  + {c} \max_{0 \le t \le T} \|\nabla df\|_{L^4(\Sigma)}^4 T\\
 &\le \frac{\delta^2}{6C_1^2}
\end{align*}
if we choose $T$ small enough.
 Finally,
\begin{align*}
|g_t| &\le {c} |df|^2 |u_t| + {c} |df|^2 |f_t| + {c}|df_t| |df|
\end{align*}
and
\begin{align*}
\|g_t\|_{L^2}^2 &\le {c}\|df\|_{C^0}^4 \iint |u_t|^2 + |f_t|^2 + {c} \|df\|_{C^0}^2 \iint |df_t|^2\\
&\le {c}\|f\|_{Y}^4  \left( \max_{0 \le t \le T} \|u_t(t) \|_{L^2(\Sigma)}^2 + \max_{0 \le t \le T} \|f_t(t)\|_{L^2(\Sigma)}^2 \right) T \\
&\quad + {c} \|f\|_Y^2 \max_{0 \le t \le T} \|\nabla f_t\|_{L^2(\Sigma)}^2 T\\
&\le \frac{\delta^2}{6C_1^2}
\end{align*}
 if we choose $T$ small enough.
 
 Therefore, if we choose $T$ small enough, we have $\|g\|_X^2 \le \frac{\delta^2}{2C_1^2}$.
  Noting that $S(f)-h_0=h-h_0$ satisfies
  \[
 (\partial_t-e^{-2u}\Delta)(h-h_0)=g + (e^{-2u}-1) \Delta h_0 \qquad (h-h_0)(0)=0.
 \] 
 The bounds \eqref{A2b} give
 \begin{align*}
\|S(f)-h_0\|_Y^2\ & \le\   C_1^2 \left( \|g\|_{X}^2 + \|(e^{-2u}-1) \Delta h_0\|_{X}^2 \right)\\
&= C_1^2 \left( \|g\|_X^2 +  \|(e^{-2u}-1) \partial_t h_0\|_{X}^2 \right)
\end{align*}
because $h_0$ satisfies \eqref{Ah_0eq}.

Now by \Cref{P0Lemma} with $h=h_0$, $u_1=u$, $u_2=0$,
\begin{equation*}
\|(e^{-2u}-1) \partial_t h_0\|_{X} \le \frac{C_3}{2C_1}\|u\|_Z \le \frac{\delta}{2C_1}.
\end{equation*}
 
 This implies
 \[
 \|S(f)-h_0\|_Y^2\ \le\   C_1^2 \left( \frac{\delta^2}{2C_1^2} +\frac{\delta^2}{4C_1^2} \right) \le \delta^2.
 \]
  Therefore $S(f)\in B$.
  \end{proof}

\begin{lemma}\label{P1Lemma2}
Fix $f_0 \in W^{3,2}(\Sigma)$ and $u \in B'$.
Then there is an $T_0 = T_0(C_0,\delta,\delta')>0$ such that for all $T \le T_0$ and for each $f_1,f_2\in B$,   
 \begin{align}\label{AT}
\|S_1(f_1,u)-S_1(f_2,u)\|_Y \le \frac{1}{3} \|f_1-f_2\|_Y.
 \end{align}
\end{lemma}  
  
\begin{proof}
 Set $h_i=S_1(f_i,u)$ and  $g_i=e^{-2u} A_{f_i}(df_i,df_i)$ for $i=1,2$ and subtracting, the function $h_1-h_2$ satisfies
\[
(\partial_t-e^{-2u}\Delta)(h_1-h_2)=g_1-g_2 \qquad (h_1-h_2)(0)= 0.
\]
Hence \eqref{A2} gives a bound
\begin{align}\label{A7}
\|h_1-h_2\|_Y \le C_1 \|g_1-g_2\|_{X}.
\end{align}
Next, we have
\begin{align*}
 g_1-g_2\ &= \  e^{-2u} (A_{f_1}-A_{f_2})(df_1,df_1) + e^{-2u} A_{f_2}(df_1+df_2,df_1-df_2) \\
 &= I + II.
 \end{align*}
 So,  there is a constant ${c}$ with
 \[
 |g_1-g_2|^2\ \le \ {c} |{f_1}-f_2|^2 |df_1|^4 +  {c} |df_1-df_2|^2 \left(|df_1|^2+|df_2|^2\right).
 \]
 Integrating and applying  Holder's inequality,   \eqref{A1}, and \eqref{Aballbound} gives
 \begin{align*}
\|g_1-g_2\|_{L^2}^2 & \le   {c}  \|{f_1}-f_2\|_{C^0}^2 \iint |df_1|^4 \ +\ {c} \|df_1-df_2\|_{L^4}^2 \left( \|df_1\|_{L^4}^2 +\|df_2\|_{L^4}^2 \right) \\
& \le    {c}  \|{f_1}-f_2\|_Y^2   \|f_1\|_Y^4 |\Sigma| T  \ +\ {c} \|{f_1}-f_2\|_Y^2 (\|f_1\|_Y^2 + \|f_2\|_Y^2) |\Sigma|^{1/2} T^{1/2}\\
& \le    \frac{1}{27C_1^2}\,   \|{f_1}-f_2\|_Y^2
  \end{align*}
  if we choose $T$ small enough.

For $\nabla^2 (g_1-g_2)$, first note that
\begin{align*}
\nabla (A_{f_1}-A_{f_2}) &= DA_{f_1} df_1 - DA_{f_2} df_2 = (DA_{f_1} - DA_{f_2}) df_1 + DA_{f_2} (df_1-df_2)\\
\nabla (DA_{f_1} - DA_{f_2}) & = (D^2 A_{f_1} - D^2 A_{f_2})df_1 + D^2 A_{f_2}(df_1-df_2).
\end{align*}
So, we get
\begin{align*}
|\nabla^2 I| &\le {c}|f_1-f_2| |df_1|^2 |\nabla^2 u| + {c} |f_1-f_2| |df_1|^2 |\nabla u|^2 + {c} |f_1-f_2| |df_1|^3 |\nabla u| \\
&\qquad + {c} |df_1-df_2| |df_1|^2 |\nabla u| + {c} |f_1-f_2| |\nabla df_1| |df_1| |\nabla u|\\
&\quad + {c} |f_1-f_2||df_1|^4 + {c}|f_1-f_2| |\nabla df_1| |df_1|^2\\
&\qquad  + {c}|df_1-df_2||df_2||df_1|^2 + {c} |\nabla df_1-\nabla df_2| |df_1|^2\\
&\qquad + {c}|df_1-df_2| |\nabla df_1| |df_1|\\
&\quad + {c} |f_1-f_2| |\nabla^3 f_1| |df_1| + {c} |f_1-f_2| |\nabla df_1|^2.
\end{align*}
Using \eqref{A1}, \eqref{A1'}, \eqref{A1''}, \eqref{AZ2}, \eqref{AZ3}, and using Young's inequality, we can estimate it term by term.
\begin{align*}
\iint |f_1-f_2|^2 |df_1|^4 |\nabla^2 u|^2 &\le \|f_1-f_2\|_Y^2 \|f_1\|_Y^4 \max_{0 \le t \le T} \|\nabla^2 u(t)\|_{L^2(\Sigma)}^2 T\\
\iint |f_1-f_2|^2 |df_1|^4 |\nabla u|^4 &\le \|f_1-f_2\|_Y^2 \|f_1\|_Y^4 \max_{0 \le t \le T} \|\nabla u(t)\|_{L^4(\Sigma)}^4 T\\
\iint |f_1-f_2|^2 |df_1|^6 |\nabla u|^2 &\le \|f_1-f_2\|_Y^2 \|f_1\|_Y^6 \max_{0 \le t \le T} \|\nabla u(t)\|_{L^2(\Sigma)}^2 T\\
\iint |f_1-f_2|^2 |\nabla df_1|^2 |df_1|^2 |\nabla u|^2 &\le \|f_1-f_2\|_Y^2 \|f_1\|_Y^4 \max_{0 \le t \le T} \|\nabla u(t)\|_{L^4(\Sigma)}^4 T\\
&\quad + \|f_1-f_2\|_Y^2 \max_{0 \le t \le T} \|\nabla df_1(t)\|_{L^4(\Sigma)}^4 T\\
\end{align*}
\begin{align*}
\iint |f_1-f_2|^2 |df_1|^8 &\le \|f_1-f_2\|_Y^2 \|f_1\|_Y^8 |\Sigma| T\\
\iint |f_1-f_2|^2 |\nabla df_1|^2 |df_1|^4 &\le \|f_1-f_2\|_Y^2 \|f_1\|_Y^4 \max_{0 \le t \le T} \|\nabla df_1(t)\|_{L^2(\Sigma)}^2 T\\
\iint |df_1-df_2|^2 |df_2|^2 |df_1|^4 &\le \|f_1-f_2\|_Y^2 \|f_2\|_Y^2 \|f_1\|_Y^4 |\Sigma| T\\
\iint |\nabla df_1 - \nabla df_2|^2 |df_1|^4 &\le \|f_1\|_Y^4 \max_{0 \le t \le T} \|\nabla df_1(t) - \nabla df_2(t)\|_{L^2(\Sigma)}^2 T\\
\iint |df_1-df_2|^2 |\nabla df_1|^2 |df_1|^2 &\le \|f_1-f_2\|_Y^2 \|f_1\|_Y^2 \max_{0 \le t \le T} \|\nabla df_1(t)\|_{L^2(\Sigma)}^2 T\\
\iint |f_1-f_2|^2 |\nabla^3 f_1|^2 |df_1|^2 &\le \|f_1-f_2\|_Y^2 \|f_1\|_Y^2 \max_{0 \le t \le T} \|\nabla^3 f_1(t)\|_{L^2(\Sigma)}^2 T\\
\iint |f_1-f_2|^2 |\nabla df_1|^4 &\le \|f_1-f_2\|_Y^2 \max_{0 \le t \le T} \|\nabla df_1(t)\|_{L^4(\Sigma)}^4 T.
\end{align*}

Hence, using \eqref{Aballbound}, if we choose $T$ small enough, we get
\begin{align*}
\|\nabla^2 I\|_{L^2}^2 & \le \frac{1}{54C_1^2} \|f_1-f_2\|_Y^2.
\end{align*}
We obtain similar result for $II$ if we choose $T$ small enough:
\begin{align*}
\|\nabla^2 II\|_{L^2}^2 & \le \frac{1}{54C_1^2} \|f_1-f_2\|_Y^2.
\end{align*}
Hence, we obtain that $\|\nabla^2 (g_1-g_2)\|_{L^2}^2 \le \frac{1}{27C_1^2} \|f_1-f_2\|_Y^2$.

Finally, compute $\partial_t (g_1-g_2)$.
As above, note that
\[
\partial_t (A_{f_1}-A_{f_2}) = DA_{f_1} \partial_t f_1 - DA_{f_2} \partial_t f_2 = (DA_{f_1}-DA_{f_2}) \partial_t f_1 + DA_{f_2} (\partial_t (f_1-f_2)).
\]
So,
\begin{align*}
|\partial_t (g_1-g_2)| &\le {c}|f_1-f_2| |df_1|^2 |u_t| + {c}|f_1-f_2| |\partial_t f_1| |df_1|^2 + {c}|\partial_t (f_1-f_2)| |df_1|^2\\
&\quad  + {c}|f_1-f_2| |\partial_t df_1| |df_1|  \\
&\quad + {c} |df_1+df_2| |df_1-df_2| |u_t| +{c} |\partial_t f_2| |df_1+df_2| |df_1-df_2| \\
&\quad  + {c}|\partial_t (df_1+df_2)| |df_1-df_2| + {c} |df_1+df_2| |\partial_t (df_1-df_2)|.
\end{align*}
Similar with above, by \eqref{A1}, \eqref{A1'}, \eqref{A1''}, \eqref{AZ2}, \eqref{AZ3} and \eqref{Aballbound},
\begin{align*}
\|\partial_t& (g_1-g_2)\|_{L^2}^2 \le {c} \|f_1\|_Y^4 \|f_1-f_2\|_Y^2 \left( \max_{0 \le t \le T} \|u_t(t)\|_{L^2(\Sigma)}^2 + \max_{0 \le t \le T} \|\partial_t f_1(t)\|_{L^2(\Sigma)}^2 \right) T \\
&\,\, + {c} \|f_1\|_Y^4 \max_{0 \le t \le T} \|\partial_t (f_1(t)-f_2(t))\|_{L^2(\Sigma)}^2 T\\
&\,\, + {c} \|f_1\|_Y^2 \|f_1-f_2\|_Y^2 \max_{0 \le t \le T} \|\partial_t df_1(t) \|_{L^2(\Sigma)}^2 T\\
&\,\, + {c} (\|f_1\|_Y^2 + \|f_2\|_Y^2) \|f_1-f_2\|_Y^2  \left( \max_{0 \le t \le T} \|u_t(t)\|_{L^2(\Sigma)}^2 + \max_{0 \le t \le T} \|\partial_t f_2(t)\|_{L^2(\Sigma)}^2 \right) T\\
&\,\, + {c} \|f_1-f_2\|_Y^2 \left(\max_{0 \le t \le T} \|\partial_t df_1(t)\|_{L^2(\Sigma)}^2 + \max_{0 \le t \le T} \|\partial_t df_2(t)\|_{L^2(\Sigma)}^2 \right) T\\
&\,\, + {c} (\|f_1\|_Y^2 + \|f_2\|_Y^2)  \max_{0 \le t \le T} \|\partial_t (df_1(t)-df_2(t))\|_{L^2(\Sigma)}^2 T\\
&\le \frac{1}{27C_1^2} \|f_1-f_2\|_Y^2
\end{align*}
if we choose $T$ small enough.

Combine all of them, 
 \[
 \|h_1-h_2\|_Y \le C_1 \|g_1-g_2\|_X \le \frac{1}{3} \|f_1-f_2\|_Y
 \]
 which proves the lemma.
 \end{proof}

 \begin{lemma}\label{P1Lemma3}
 Fix $f_0 \in W^{3,2}(\Sigma)$ and $f \in B$.
Then there is an $T_0 = T_0(C_0,\delta,\delta')>0$ such that for all $T \le T_0$ and for each $u_1,u_2\in B'$,   
 \begin{align}\label{AT2}
\|S_1(f,u_1)-S_1(f,u_2)\|_Y \le \frac{C_3}{2} \|u_1-u_2\|_{Z}.
 \end{align}
\end{lemma}  
 
 \begin{proof}
Set $h_i=S_1(f,u_i)$.
Multiplying $e^{2u_i}$ to the equation for $h_i$ respectively and subtracting them gives
\begin{align*}
e^{2u_1} \partial_t (h_1 - h_2) - \Delta (h_1-h_2) &= -(e^{2u_1}-e^{2u_2}) \partial_t h_2\\
(\partial_t - e^{-2u_1} \Delta) (h_1-h_2) &= (e^{2u_2-2u_1}-1) \partial_t h_2.
\end{align*}

So, $h_1-h_2$ satisfies the estimate from \eqref{A2b}, and by \Cref{P0Lemma},
\[
\|h_1-h_2\|_Y \le C_1 \|(e^{2u_2-2u_1}-1) \partial_t h_2\|_X \le \frac{C_3}{2} \|u_1-u_2\|_Z
\]
if we choose $T$ small enough.

 \end{proof}


\subsection{The construction $S_2$}  Define an operator
\[
S_2:  Y_T \times Z_T \to Z_T
\]
by $S_2(f,u)=v$ where $v\in Z_{T}$ is the unique solution of
\begin{align}\label{A3'}
\partial_t v =b|df|^2 e^{-2u} - a \qquad v(0)=0.
\end{align}


\begin{lemma}
In the above definition, $v \in Z_T$.
\end{lemma}

\begin{proof}
 From \eqref{A3'}, we directly get
 \begin{equation}
 v(t) = \int_{0}^{t} (b|df|^2 e^{-2u} - a).
 \end{equation}
 So, $\|v\|_{L^2}$ and $\|v_t\|_{L^2}$ is trivially bounded if $f \in Y_T$ and $u \in Z_T$.
(Because $u \in Z_T$, we have $e^{-2u}$ is pointwise uniformly bounded by $e^{C\|u\|_Z}$.)
 Applying Cauchy-Schwarz, we obtain the pointwise bound
\begin{align*}
|\nabla^3 v|^2 &=  \left| b \int_{0}^{T} \nabla^2 \left( \langle \nabla df, df \rangle e^{-2u} -2 |df|^2 e^{-2u} \nabla u  \right)\right|^2\\
&\le {c} T \int_{0}^{T} \Big( |\nabla^4 f|^2 |df|^2 + |\nabla^3 f|^2 |\nabla df|^2  + |\nabla^3 f|^2 |df|^2 |\nabla u|^2 + |\nabla^2 f|^4 |\nabla u|^2 \\
&\qquad  + |\nabla^2 f|^2 |df|^2 (|\nabla u|^4 + |\nabla^2 u|^2) + |df|^4 (|\nabla u|^6 + |\nabla^2 u|^2 |\nabla u|^2 + |\nabla^3 u|^2) \Big)
\end{align*}
so
\begin{align*}
\|\nabla^3 v\|_{L^2}^2 &\le {c} T^2 \|df\|_{C^0}^2 \|\nabla^4 f\|_{L^2}^2 + {c} T^2 \|\nabla^3 f\|_{L^4}^2 \|\nabla^2 f\|_{L^4}^2 \\
&\quad + {c} T^2 \|df\|_{C^0}^2 \|\nabla^3 f\|_{L^4}^2 \|\nabla u\|_{L^4}^2 + {c} T^2 \|\nabla^2 f\|_{L^8}^4 \|\nabla u\|_{L^4}^2\\
&\quad + {c} T^2 \|df\|_{C^0}^2 \|\nabla^2 f\|_{L^4}^2 (\|\nabla u\|_{L^8}^4 + \|\nabla^2 u\|_{L^4}^2)\\
&\quad + {c} T^2 \|df\|_{C^0}^4 (\|\nabla u\|_{L^6}^6 + \|\nabla^2 u\|_{L^4}^2 \|\nabla u\|_{L^4}^2 + \|\nabla^3 u\|_{L^2}^2)
\end{align*}
which is bounded if $f \in Y_T$ and $u \in Z_T$.

Finally, compute $\nabla v_t$ from \eqref{A3'}.
\begin{align*}
|\nabla v_t|^2 &\le {c} |\nabla df| |df| + |df|^2 |\nabla u|\\
\|\nabla v_t\|_{L^2}^2 &\le {c} \|df\|_{C^0}^2 \|\nabla df\|_{L^2}^2 + {c} \|df\|_{C^0}^4 \|\nabla u\|_{L^2}^2
\end{align*}
which is bounded if $f \in Y_T$ and $u \in Z_T$.
\end{proof}

In fact, we can show further.

 \begin{lemma}\label{P2Lemma1}
 Fix  $f_0\in W^{3,2}(\Sigma)$.
  Then there is $T_0 = T_0(C_0,\delta,\delta') >0$ such that for all $T \le T_0$,
$S_2$ restricts to an operator $S_2 :  B \times B' \to  B'$.
\end{lemma}
 
 \begin{proof}
From previous calculation, we have
\[
\|\nabla^3 v\|_{L^2}^2 \le {c} T^2 \|f\|_Y^4 (1 +  \|u\|_Z^2 + \|u\|_Z^4 + \|u\|_Z^6).
\]
So, if we choose $T$ small enough, we get $\|\nabla^3 v\|_{L^2}^2 \le \frac{{\delta'}^2}{4}$.
Also, because $|v_t| \le {c}(\|df\|_{C^0} + 1)$ and $|v| \le {c} T (\|df\|_{C^0} + 1)$, we can make $\|v_t\|_{L^2}^2, \|v\|_{L^2}^2 \le \frac{{\delta'}^2}{4}$ if we choose $T$ small.
Finally,
\[
\|\nabla v_t\|_{L^2}^2 \le {c}\|df\|_{C^0}^2 \max_{0 \le t \le T} \|\nabla df(t)\|_{L^2(\Sigma)}^2 \, T + {c} \|df\|_{C^0}^4 \max_{0 \le t \le T} \|\nabla u(t)\|_{L^2(\Sigma)}^2 \, T
\]
so if we choose $T$ small enough, we get $\|\nabla v_t\|_{L^2}^2 \le \frac{{\delta'}^2}{4}$.
This proves the lemma.
 \end{proof}

\begin{lemma}\label{P2Lemma2}
Fix $f_0 \in W^{3,2}(\Sigma)$ and $u \in B'$.
Then there is an $T_0 = T_0(C_0,\delta,\delta')>0$ such that for all $T \le T_0$ and for each $f_1,f_2\in B$,   
 \begin{align}\label{AT'}
\|S_2(f_1,u)-S_2(f_2,u)\|_Z \le T^{1/4} \|f_1-f_2\|_Y.
 \end{align}
\end{lemma}  
  
 \begin{proof}
 Set $v_i = S_2(f_i,u)$.
 Then from \eqref{A3'}, subtracting them gives
\begin{align*}
(v_1-v_2)_t &= b(|df_1|^2 - |df_2|^2) e^{-2u} = b e^{-2u} \langle df_1+df_2, df_1-df_2 \rangle\\
v_1-v_2 &= b \int_{0}^{t} e^{-2u} \langle df_1+df_2, df_1-df_2 \rangle.
\end{align*}
 So,
 \begin{align*}
 \|v_1-v_2\|_{L^2}^2 &\le {c} T^2 (\|df_1\|_{C^0}^2 + \|df_2\|_{C^0}^2) \|df_1-df_2\|_{L^2}^2\\
 &\le \frac{\sqrt{T}}{4} \|f_1-f_2\|_Y^2
 \end{align*}
 if we choose $T$ small enough.
 Also,
 \begin{align*}
 \|(v_1-v_2)_t\|_{L^2}^2 &\le {c} (\|f_1\|_Y^2 + \|f_2\|_Y^2) \max_{0 \le t \le T} \|df_1(t)-df_2(t)\|_{L^2(\Sigma)}^2 \, T\\
 &\le  \frac{\sqrt{T}}{4} \|f_1-f_2\|_Y^2
 \end{align*}
 if we choose $T$ small enough.
 
 Next, compute $\nabla^3 (v_1-v_2)$.
 \begin{align*}
\nabla (v_1-v_2) &=  b \int_{0}^{t} e^{-2u} \Big(  \langle \nabla (df_1+df_2), df_1-df_2 \rangle + \langle df_1+df_2, \nabla (df_1-df_2) \rangle \\
&\qquad \qquad  - \langle df_1+df_2, df_1-df_2 \rangle 2\nabla u \Big).
\end{align*}
So,
\begin{align*}
|\nabla^3 (v_1-v_2)| &\le {c} \int_{0}^{T} e^{-2u} \Big( |\nabla^3 (df_1+df_2)| |df_1-df_2| + |\nabla^2 (df_1+df_2)| |\nabla (df_1-df_2)| \\
&\qquad + |\nabla (df_1+df_2)| |\nabla^2 (df_1-df_2)| + |df_1+df_2| |\nabla^3 (df_1-df_2)| \\
&\quad + |\nabla^2 (df_1+df_2)| |df_1-df_2 | |\nabla u| + |\nabla (df_1+df_2)| |\nabla (df_1-df_2)| |\nabla u|\\
&\qquad + |df_1+df_2| |\nabla^2 (df_1-df_2)| |\nabla u|\\
&\quad + |\nabla (df_1+df_2)| |df_1-df_2| (|\nabla u|^2 + |\nabla^2 u|)\\
&\qquad + |df_1+df_2| |\nabla (df_1-df_2)| (|\nabla u|^2 + |\nabla^2 u|)\\
&\quad + |df_1+df_2| |df_1-df_2| (|\nabla u|^3 + |\nabla^2 u| |\nabla u| + |\nabla^3 u|) \Big).
 \end{align*}
 Integrating over $\Sigma \times [0,T]$ gives
 \begin{align*}
 \|\nabla^3 &(v_1-v_2)\|_{L^2}^2\\
 &\le {c} T^2 \|df_1-df_2\|_{C^0}^2 \|\nabla^4 (f_1+ f_2)\|_{L^2}^2  + {c} T^2 \|\nabla^3 ( f_1 + f_2)\|_{L^4}^2 \|\nabla^2 (f_1-f_2)\|_{L^4}^2\\
 &\quad + {c} T^2 \|\nabla^2 (f_1+f_2)\|_{L^4}^2 \|\nabla^3 (f_1-f_2)\|_{L^4}^2 + {c} T^2 \|df_1+df_2\|_{C^0}^2 \|\nabla^4 (f_1-f_2)\|_{L^2}^2\\
 &+ {c} T^2 \|df_1-df_2\|_{C^0}^2 \|\nabla^3 (f_1+f_2)\|_{L^4}^2 \|\nabla u\|_{L^4}^2\\
 &\quad + {c} T^2 \|\nabla^2 (f_1-f_2) \|_{L^4}^2 \|\nabla^2 (f_1+f_2)\|_{L^8}^4 \|\nabla u\|_{L^8}^4\\
 &\quad + {c} T^2 \|df_1+df_2\|_{C^0}^2 \|\nabla^3 (f_1-f_2)\|_{L^4}^2 \|\nabla u\|_{L^4}^2\\
 &+ {c} T^2 \|df_1-df_2\|_{C^0}^2 \|\nabla^2 (f_1+f_2)\|_{L^4}^2 (\|\nabla u\|_{L^8}^4 + \|\nabla^2 u\|_{L^4}^2)\\
 &\quad + {c} T^2 \|df_1+df_2\|_{C^0}^2 \|\nabla^2 (f_1-f_2)\|_{L^4}^2 (\|\nabla u\|_{L^8}^4 + \|\nabla^2 u\|_{L^4}^2)\\
 &+ {c} T^2 \|df_1-df_2\|_{C^0}^2 \|df_1+df_2\|_{C^0}^2 ( \|\nabla u\|_{L^6}^6 + \| \nabla^2 u \|_{L^4}^2 \|\nabla u\|_{L^4}^2 + \|\nabla^3 u\|_{L^2}^2)\\
 &\le \frac{\sqrt{T}}{4} \|f_1-f_2\|_Y^2
 \end{align*}
 if we choose $T$ small enough.
 
 Finally consider $\nabla (v_1-v_2)_t$.
 \begin{align*}
 \nabla (v_1-v_2)_t &= b e^{-2u} \left(  \langle \nabla (df_1+df_2), df_1-df_2 \rangle + \langle df_1+df_2, \nabla (df_1-df_2) \rangle \right.\\
&\qquad \qquad \left. - \langle df_1+df_2, df_1-df_2 \rangle 2\nabla u \right).
 \end{align*}
 So,
 \begin{align*}
 \|\nabla (v_1-v_2)_t\|_{L^2}^2 &\le {c} \|df_1-df_2\|_{C^0}^2 \max_{0 \le t \le T} \|\nabla^2 (f_1(t)+f_2(t))\|_{L^2(\Sigma)}^2 \, T \\
 &\quad + {c} \|df_1+df_2\|_{C^0}^2 \max_{0 \le t \le T} \|\nabla^2 (f_1(t)-f_2(t))\|_{L^2(\Sigma)}^2 \, T \\
 &\quad + {c} \|df_1+df_2\|_{C^0}^2 \|df_1-df_2\|_{C^0}^2 \max_{0 \le t \le T} \|\nabla u(t) \|_{L^2(\Sigma)}^2 \, T\\
 &\le \frac{\sqrt{T}}{4} \|f_1-f_2\|_Y^2
 \end{align*}
 if we choose $T$ small enough.
 
 In summary, we get
 \[
 \|v_1-v_2\|_Z^2 \le \sqrt{T} \|f_1-f_2\|_Y^2
 \]
 which proves the lemma.

 \end{proof}

 \begin{lemma}\label{P2Lemma3}
 Fix $f_0 \in W^{3,2}(\Sigma)$ and $f \in B$.
Then there is an $T_0 = T_0(C_0,\delta,\delta')>0$ such that for all $T \le T_0$ and for each $u_1,u_2\in B'$,   
 \begin{align}\label{AT2'}
\|S_2(f,u_1)-S_2(f,u_2)\|_Z \le \frac{1}{3} \|u_1-u_2\|_{Z}.
 \end{align}
\end{lemma}  
 
 \begin{proof}
 Set $v_i=S_2(f,u_i)$.
 Subtracting them gives
 \begin{align*}
 (v_1-v_2)_t &= b |df|^2 (e^{-2u_1}-e^{-2u_2})\\
 v_1-v_2 &= b \int_{0}^{t} |df|^2 (e^{-2u_1}-e^{-2u_2}).
 \end{align*}
 Using $|e^{-2u_1}-e^{-2u_2}| \le {c}|u_1-u_2|$, we have
 \begin{align*}
 \|v_1-v_2\|_{L^2}^2 &\le {c} T^2 \|df\|_{C^0}^4 \|u_1-u_2\|_{L^2}^2\\
 \|(v_1-v_2)_t\|_{L^2}^2 &\le {c} \|df\|_{C^0}^4 \max_{0 \le t \le T} \|u_1(t)-u_2(t)\|_{L^2(\Sigma)}^2 \, T
 \end{align*}
 so if we choose $T$ small enough, we have that $\|v_1-v_2\|_{L^2}^2, \|(v_1-v_2)_t\|_{L^2}^2 \le \frac{1}{36} \|u_1-u_2\|_Z^2$.
 
 Next, compute $\nabla^3 (v_1-v_2)$.
 \begin{align*}
 \nabla (v_1-v_2) &= b \int_{0}^{t} \Big( 2\langle \nabla df, df \rangle (e^{-2u_1}-e^{-2u_2}) - |df|^2 (e^{-2u_1}-e^{-2u_2}) 2 (\nabla u_1 - \nabla u_2) \Big).
 \end{align*}
 So,
 \begin{align*}
 |\nabla^3&  (v_1-v_2)| \le {c} \int_{0}^{T} \Big( |\nabla^4 f| |df| |u_1-u_2| + |\nabla^3 f| |\nabla^2 f| |u_1-u_2|\\
 &\quad + (|\nabla^3 f| |df| + |\nabla^2 f|^2) |u_1-u_2| |\nabla (u_1-u_2)|\\
 &\quad + |\nabla^2 f| |df| |u_1-u_2| (|\nabla (u_1-u_2)|^2 + |\nabla^2 (u_1-u_2)|)\\
 &\quad + |df|^2 |u_1-u_2| (|\nabla (u_1-u_2)|^3 + |\nabla^2 (u_1-u_2)| |\nabla (u_1-u_2)| + |\nabla^3 (u_1-u_2)|) \Big).
 \end{align*}
Now we integrate over $\Sigma \times [0,T]$.
 \begin{align*}
 \|\nabla^3& (v_1-v_2)\|_{L^2}^2\\
 &\le {c} T^2 \|df\|_{C^0}^2 \|u_1-u_2\|_{C^0}^2 \|\nabla^4 f\|_{L^2}^2 + {c} T^2 \|u_1-u_2\|_{C^0}^2 \|\nabla^3 f\|_{L^4}^2 \|\nabla^2 f\|_{L^4}^2\\
 &\quad +  {c} T^2 (\|df\|_{C^0}^2 \|\nabla^3 f\|_{L^4}^2 + \|\nabla^2 f\|_{L^8}^4) \|u_1-u_2\|_{C^0}^2 \|\nabla (u_1-u_2)\|_{L^4}^2\\
 &\quad + {c} T^2 \|df\|_{C^0}^2 \|u_1-u_2\|_{C^0}^2 \|\nabla^2 f\|_{L^4}^2 (\|\nabla (u_1-u_2)\|_{L^8}^4 + \|\nabla^2 (u_1-u_2)\|_{L^4}^2)\\
 &\quad + {c} T^2 \|df\|_{C^0}^4 \|u_1-u_2\|_{C^0}^2 \|\nabla (u_1-u_2)\|_{L^6}^6\\
 &\quad + {c} T^2 \|df\|_{C^0}^4 \|u_1-u_2\|_{C^0}^2  \|\nabla^2 (u_1-u_2)\|_{L^4}^2 \|\nabla (u_1-u_2)\|_{L^4}^2\\
 &\quad + {c} T^2 \|df\|_{C^0}^4 \|u_1-u_2\|_{C^0}^2 \|\nabla^3 (u_1-u_2)\|_{L^2}^2\\
 &\le \frac{1}{36} \|u_1-u_2\|_Z^2
 \end{align*}
 if we choose $T$ small enough.
 
 Finally,
 \begin{align*}
 |\nabla (v_1-v_2)_t| &\le {c} |\nabla df| |df| |u_1-u_2| + |df|^2 |u_1-u_2|  |\nabla (u_1 - u_2)|
 \end{align*}
 so
 \begin{align*}
 \|\nabla (v_1-v_2)_t\|_{L^2}^2 &\le {c} \|df\|_{C^0}^2 \|u_1-u_2\|_{C^0}^2 \max_{0 \le t \le T} \|\nabla^2 f(t) \|_{L^2(\Sigma)}^2 \, T\\
 &\quad + {c} \|df\|_{C^0}^4 \|u_1-u_2\|_{C^0}^2 \max_{0 \le t \le T} \|\nabla (u_1(t)-u_2(t))\|_{L^2(\Sigma)}^2 \,  T\\
 &\le \frac{1}{36} \|u_1-u_2\|_Z^2
 \end{align*}
 if we choose $T$ small enough.
 
 In summary, we get
 \[
 \|v_1-v_2\|_Z^2 \le \frac{1}{9} \|u_1-u_2\|_Z^2
 \]
 which proves the lemma.

 \end{proof}


 \section{Existence of fixed point}
 \label{sec5}
 
 Because $Y_T$ and $Z_T$ are Hilbert space, $Y_T \times Z_T$ is also a Hilbert space and we can equip the norm
 \begin{equation}
 \|(f,u)\|_{Y \times Z} = (C_3)^{-1} \|f\|_Y + \|u\|_Z.
 \end{equation}
 
 Define an operator $\mathcal{S} : Y_T \times Z_T \rightarrow Y_T \times Z_T$ by
 \begin{equation} \label{the map S}
 \mathcal{S} (f,u) = (S_1(f,u),S_2(f,u)).
 \end{equation}
 
 \begin{prop}\label{Fixed pt}
 Fix $f_0 \in W^{3,2}(\Sigma)$.
 Then there is an $T_0 = T_0(C_0,\delta,\delta')>0$ such that for all $T \le T_0$,
  \begin{enumerate}
\item[(a)]  $\mathcal{S}$ restricts to an operator $\mathcal{S}: B \times B' \to B \times B'$.
\item [(b)] For each $f_1,f_2 \in B$ and $u_1,u_2 \in B'$,   
 \begin{align}\label{Contraction}
\|\mathcal{S}(f_1,u_1)-\mathcal{S}(f_2,u_2)\|_{Y \times Z} \le \frac{5}{6} \|(f_1,u_1)-(f_2-u_2)\|_{Y \times Z}.
 \end{align}
\end{enumerate}
 \end{prop}
 
 \begin{proof}
 By \Cref{P1Lemma1} and \Cref{P2Lemma1}, (a) is proved.
 For (b), using Lemmas \ref{P1Lemma2}, \ref{P1Lemma3}, \ref{P2Lemma2}, \ref{P2Lemma3}, there is $T_0 = T_0(\delta,\delta')>0$ such that for all $T \le T_0$,
 \begin{align*}
 \|\mathcal{S}(f_1,u_1)&-\mathcal{S}(f_2,u_2)\|_{Y \times Z}\\
 & = (C_3)^{-1}\|S_1(f_1,u_1)- S_1(f_2,u_2) \|_Y  + \|S_2(f_1,u_1) - S_2(f_2,u_2)\|_{Z}\\
 &\le (C_3)^{-1} \|S_1(f_1,u_1)-S_1(f_2,u_1)\|_Y + (C_3)^{-1} \|S_1(f_2,u_1)-S_1(f_2,u_2)\|_Y\\
 &\quad + \|S_2(f_1,u_1) - S_2(f_2,u_1)\|_Z + \|S_2(f_2,u_1)-S_2(f_2,u_2)\|_Z\\
 &\le \frac{1}{3} (C_3)^{-1} \|f_1-f_2\|_Y + \frac{1}{2} \|u_1-u_2\|_Z\\
 &\quad + T^{1/4} \|f_1-f_2\|_Y + \frac{1}{3} \|u_1-u_2\|_Z\\
 &\le \frac{5}{6} \left( (C_3)^{-1} \|f_1-f_2\|_Y + \|u_1-u_2\|_Z \right)\\
 &= \frac{5}{6}\|(f_1,u_1)-(f_2,u_2)\|_{Y \times Z}
 \end{align*}
 if $T^{1/4} \le \frac{1}{2} (C_3)^{-1}$.
 \end{proof}
 
\begin{theorem}\label{short time}
(Short time existence for strong solution)
There is $T_0>0$ such that there exists a smooth solution $(f,u) \in B \times B'$ of \eqref{main0} on $\Sigma \times [0,T_0]$.
\end{theorem}
  
\begin{proof}
The existence of solution $f,u$ comes from \Cref{Fixed pt}.
The fact $f(\Sigma \times [0,T_0]) \subset N$ can be easily shown using nearest point projection, see for example \cite{LW08}.
Moreover, the operator $\partial_t - e^{-2u}\Delta$ is uniformly parabolic, so $|(\partial_t - e^{-2u}\Delta)f| \in L^p(\Sigma \times [0,T_0])$ for any $1 \le p < \infty$, by standard parabolic theory.
This implies
\begin{equation*}
\nabla^2 f, \partial_t f \in L^p(\Sigma \times [0,T_0])
\end{equation*}
for any $1 \le p < \infty$.

Next, by direct computation from \eqref{main0-2}, we have
\begin{equation*}
e^{2u} = e^{-2at} \left( 1 + 2b \int_{0}^{t} e^{2as}|df|^2 \right)
\end{equation*}
hence
\begin{align*}
\nabla u &= e^{-2u-2at} 2b \int_{0}^{t} e^{2as} \langle \nabla df,df \rangle\\
\int |\nabla u|^p &\le (4b)^p \int \left( \int_{0}^{t} |\nabla df| |df| \right)^p\\
&\le (4b)^p t^{p-1} \int \int_{0}^{t} |\nabla df|^p |df|^p
\end{align*}
which implies $\nabla u \in L^{p}(\Sigma \times [0,T_0])$ for any $1 \le p < \infty$.
Now taking $\nabla$ in the equation \eqref{main0-1} to get
\begin{equation*}
|(\partial_t - e^{-2u} \Delta ) \nabla f| \le C \left( |\nabla u| |\Delta f| + |\nabla u| |df|^2 + |\nabla df| |df| + |df|^3 \right) \in L^p(\Sigma \times [0,T_0])
\end{equation*}
which implies
\begin{equation*}
\nabla^3 f, \partial_t (\nabla f) \in L^p(\Sigma \times [0,T_0])
\end{equation*}
for any $1 \le p < \infty$.

Finally, from Sobolev embedding, we have $f,df \in C^{\alpha}(\Sigma \times [0,T_0])$ for some $\alpha>0$.
This implies $(\partial_t - e^{-2u}\Delta)f \in C^{\alpha,\alpha/2}(\Sigma \times [0,T_0])$ where $C^{\alpha,\alpha/2}$ is parabolic H{\"o}lder space of exponent $\alpha$.
Now by Schauder estimate and standard bootstrapping argument, we conclude that $f$ is smooth, so $u$ is.
\end{proof}

 \section{Local estimate}
 \label{sec6}

To get global weak solution, we will follow Struwe's idea:
Run the flow until singularity occurs.
Then take weak limit as new initial condition, run the flow again.
Keep going this process and we will have only finitely many singularities due to finiteness of the energy.
Because our flow is coupled, we need to re-establish the whole process with $f$ and $u$.
And this requires some condition on $b$, which can be interpreted as the sensitiveness of the conformal evolution of the metric with respect to high energy density.
Let $C_N>0$ be a constant only depending on the embedding $N \hookrightarrow \R^{L}$ such that $\|R^N\|,\|A\|,\|DA\| \le C_N$ where $R^N$ is the Riemannian curvature tensor of $N$.
And from now on, assume $b \ge C_N^2$.

\subsection{Energy estimate}

Now we establish local energy estimate.
Fix $B_{2r}$ and let $\varphi$ be a cut-off function supported on $B_{2r}$ such that $\varphi \equiv 1$ on $B_{r}$, $0 \le \varphi \le 1$ and $|\nabla \varphi| \le \frac{4}{r}$.

\begin{prop} \label{energy est}
For solutions $(f,u)$ of \eqref{main0}, we have
\begin{equation} \label{2 estimate}
\begin{aligned}
\int_{t_1}^{t_2} \int_{B_{2r}} e^{2u}|f_t|^2 \varphi^2 +& \int_{B_{2r}}|df|^2 \varphi^2 (t_2) - \int_{B_{2r}}|df|^2 \varphi^2 (t_1)\\
&\le \frac{4^2}{ar^2} (e^{2at_2}-e^{2at_1}) E_0.
\end{aligned}
\end{equation}
Especially, we have
\begin{equation}\label{energy difference}
E(B_{r},t_2) - E(B_{2r},t_1) \le \frac{4^2}{2ar^2}(e^{2at_2}-e^{2at_1})E_0.
\end{equation}
\end{prop}

\begin{proof}
From the equation \eqref{main0-1}, multiplying $e^{2u} f_t \varphi^2$ gives
\begin{align*}
\int_{B_{2r}}e^{2u}|f_t|^2 \varphi^2 &= \int_{B_{2r}}\langle f_t, \tau(f) \varphi^2 \rangle\\
& = -\int_{B_{2r}} \langle df_t, df \varphi^2 \rangle - 2 \int_{B_{2r}} \langle f_t, f_i \rangle \varphi \nabla_i \varphi\\
& \le -\frac{1}{2} \frac{d}{dt} \int_{B_{2r}} |df|^2 \varphi^2 + \frac{1}{2}\int_{B_{2r}} e^{2u} |f_t|^2 \varphi^2 +2 \int_{B_{2r}} e^{-2u} |df|^2 |\nabla \varphi|^2.
\end{align*}
So, we have
\begin{align*}
\int_{B_{2r}}e^{2u} |f_t|^2 \varphi^2 + \frac{d}{dt} \int_{B_{2r}} |df|^2 \varphi^2 &\le 4\int_{B_{2r}} e^{-2u}|df|^2 |\nabla \varphi|^2\\
&\le 4\frac{4^2}{r^2}e^{2at} \int_{B_{2r}} |df|^2\\
&\le 4\frac{4^2}{r^2}e^{2at} 2E_0.
\end{align*}
Integrating from $t_1$ to $t_2$ gives the result.
\end{proof}

\begin{lemma} \label{f_t^2 estimates}
Furthermore, assume
\begin{equation*}
\sup_{t_1 \le t \le t_2}E(B_{2r},t) < \varepsilon_1.
\end{equation*}
Then we have
\begin{align}
\int_{t_1}^{t_2}\int_{B_{2r}}e^{2u}|f_t|^2 \varphi^2 &\le 4^2 \varepsilon_1 \left( 1 +\frac{e^{2at_2}-e^{2at_1}}{2ar^2}\right) \label{e^2u f_t^2 estimate}\\
\int_{t_1}^{t_2}\int_{B_{2r}} |f_t|^2 \varphi^2 & \le e^{2at_2} 4^2 \varepsilon_1 \left( 1 + \frac{e^{2at_2}-e^{2at_1}}{2ar^2} \right). \label{f_t^2 estimate}
\end{align}
\end{lemma}

\begin{proof}
The first equation directly comes from \eqref{2 estimate}, by changing $E_0$ to $\varepsilon_1$.
Also, it is easy to see that
\begin{align*}
\int_{t_1}^{t_2}\int_{B_{2r}} |f_t|^2 \varphi^2 & =\int_{t_1}^{t_2}\int_{B_{2r}} e^{-2u} e^{2u}|f_t|^2 \varphi^2 \le e^{2at_2} \int_{t_1}^{t_2}\int_{B_{2r}} e^{2u} |f_t|^2 \varphi^2\\
& \le e^{2at_2} 4^2 \varepsilon_1 \left( 1 + \frac{e^{2at_2}-e^{2at_1}}{2ar^2} \right).
\end{align*}
\end{proof}

\subsection{Estimate for $\int |f_t|^2$}

The next step is to get estimate for derivative of $\int_{B_{2r}}|f_t|^2 \varphi^2$, which will lead to the control of itself.
For the future purpose, here we introduce more general version of it.
For now, we need $p=0$.

\begin{prop}
Let $(f,u)$ are solutions of \eqref{main0}.
For $p \ge 0$, we have
\begin{equation}\label{p+2 derivative}
\begin{aligned}
\frac{d}{dt}\int_{B_{2r}} e^{2u}|f_t|^{p+2} \varphi^2 &\le 2a(p+1)\int_{B_{2r}}e^{2u}|f_t|^{p+2}\varphi^2 + 4(p+2) \int_{B_{2r}}|f_t|^{p+2}|\nabla \varphi|^2\\
&- \frac{p+2}{4} \int_{B_{2r}}|\nabla f_t|^{2}|f_t|^{p} \varphi^2\\
&+ \left( (p+2)C_N + \frac{(p+2)C_N^2}{2} - 2b(p+1) \right) \int_{B_{2r}}|df|^2 |f_t|^{p+2}\varphi^2.
\end{aligned}
\end{equation}
Especially, we have
\begin{equation}\label{p+2 derivative result}
\begin{aligned}
\int_{B_{2r}}e^{2u}|f_t|^{p+2}\varphi^2 (t) &\le e^{2a(p+1)(t-t_0)}\\
&\quad \cdot \left( \int_{B_{2r}}e^{2u}|f_t|^{p+2}\varphi^2(t_0) + 4(p+2) \int_{t_0}^{t} \int_{B_{2r}}|f_t|^{p+2}|\nabla \varphi|^2 \right).
\end{aligned}
\end{equation}
\end{prop}

\begin{proof}
By taking time-derivative to \eqref{main0-1}, we have
\begin{equation*}
(e^{2u}f_t)_t = \Delta f_t + A(df,df)_t.
\end{equation*}
Taking inner product with $f_t |f_t|^p \varphi^2$ and integrating gives
\begin{align*}
\int \langle (e^{2u}f_t)_t,f_t |f_t|^p \varphi^2 \rangle &= \int \langle \Delta f_t, f_t |f_t|^p \varphi^2 \rangle + \int \langle A(df,df)_t,f_t |f_t|^p \varphi^2 \rangle\\
&= - \int |\nabla f_t|^2 |f_t|^p \varphi^2 - \int \langle \nabla f_t,f_t \rangle p |f_t|^{p-2} \varphi^2 \langle \nabla f_t,f_t \rangle\\
&\qquad - 2 \int \langle \nabla f_t,f_t \rangle |f_t|^p \varphi \nabla \varphi + \int \langle DA(df,df) \cdot f_t,f_t |f_t|^p \varphi^2 \rangle\\
&\qquad + \int \langle A(df_t,df),f_t |f_t|^p \varphi^2 \rangle\\
&= -\int |\nabla f_t|^2 |f_t|^p \varphi^2 - p\int |\langle \nabla f_t,f_t \rangle|^2 |f_t|^{p-2}\varphi^2\\
&\qquad + III + IV + V.
\end{align*}
Now we have
\begin{align*}
III &\le \frac{1}{4} \int |\nabla f_t|^2 \varphi^2 |f_t|^{p} + 4 \int |f_t|^{p+2} |\nabla \varphi|^2\\
IV &\le C_N \int |df|^2 |f_t|^{p+2} \varphi^2\\
V &\le C_N \int |\nabla f_t| |df| |f_t|^{p+1} \varphi^2\\
&\le \frac{1}{2} \int |\nabla f_t|^2 \varphi^2 |f_t|^p + \frac{C_N^2}{2} \int |df|^2 |f_t|^{p+2} \varphi^2.
\end{align*}
On the other hand, LHS becomes
\begin{align*}
\int \langle (e^{2u}f_t)_t,f_t |f_t|^p \varphi^2 \rangle &= \frac{1}{p+2}\frac{d}{dt}\int e^{2u}|f_t|^{p+2} \varphi^2 + 2\frac{p+1}{p+2} \int e^{2u}|f_t|^{p+2} u_t \varphi^2\\
&= \frac{1}{p+2}\frac{d}{dt}\int e^{2u}|f_t|^{p+2} \varphi^2 + 2b\frac{p+1}{p+2} \int |df|^2|f_t|^{p+2} \varphi^2\\
&\qquad - 2a \frac{p+1}{p+2} \int e^{2u}|f_t|^{p+2}\varphi^2.
\end{align*}
All together, we have
\begin{align*}
\frac{d}{dt}\int e^{2u}|f_t|^{p+2} \varphi^2 &\le 2a(p+1) \int e^{2u}|f_t|^{p+2}\varphi^2 + 4(p+2) \int |f_t|^{p+2}|\nabla \varphi|^2\\
&\qquad - \frac{p+2}{4} \int |\nabla f_t|^2 |f_t|^p \varphi^2\\
&\qquad + \left((p+2)C_N + \frac{(p+2)C_N^2}{2} - 2b(p+1) \right)\int |df|^2 |f_t|^{p+2}\varphi^2.
\end{align*}
By the choice of $b$, the last term is negative for all $p \ge 0$.
Hence, 
\begin{align*}
\frac{d}{dt}\int e^{2u}|f_t|^{p+2} \varphi^2 &\le 2a(p+1) \int e^{2u}|f_t|^{p+2} \varphi^2 + 4(p+2) \int |f_t|^{p+2} |\nabla \varphi|^2\\
\int_{B_{2r}} e^{2u}|f_t|^{p+2} \varphi^2(t) &\le e^{2a(p+1)(t-t_0)}\\
&\quad \cdot \left( \int_{B_{2r}}e^{2u}|f_t|^{p+2} \varphi^2 (t_0) + 4(p+2) \int_{t_0}^{t} \int_{B_{2r}}|f_t|^{p+2} |\nabla \varphi|^2 \right)
\end{align*}
by Gronwall's inequality.
\end{proof}

\begin{lemma} \label{e^2u f_t^2 t lemma}
Let $(f,u)$ are solutions of \eqref{main0}.
Assume that
\begin{equation*}
\sup_{T-2\delta r^2 \le t \le T}E(B_{2r},t) < \varepsilon_1.
\end{equation*}
Then for $t \in [T-\delta r^2,T]$, we have
\begin{equation} \label{e^2u f_t^2 estimate t}
\int_{B_{2r}} e^{2u} |f_t|^2 \varphi^2 (t) \le C_1(r,\delta,t)C_2(r,\delta,t) \varepsilon_1
\end{equation}
where
\begin{align}
C_1(r,\delta,t) &= 4^2 \left( 1 + e^{2at} \frac{1-e^{-2a\delta r^2}}{2ar^2} \right)\\
C_2(r,\delta,t) &= e^{6a\delta r^2} \left( \frac{1}{\delta r^2} + \frac{16(4)^2}{r^2}e^{2at} \right).
\end{align}
\end{lemma}

\begin{proof}
Suppose $\varphi$ be a cut-off function supported on $B_{3r/2}$ and $\varphi \equiv 1$ on $B_{r}$ and $|\nabla \varphi| \le \frac{4}{r}$.
Also, let $\psi$ be a cut-off function supported on $B_{2r}$ and $\psi \equiv 1$ on $B_{3r/2}$ and $|\nabla \psi| \le \frac{4}{r}$.
From \eqref{p+2 derivative result} for $p=0$ and using \eqref{f_t^2 estimate}, we have
\begin{align*}
\int_{B_{2r}} e^{2u}|f_t|^{2} \varphi^2(t) &\le e^{2a(t-t_0)} \left( \int_{B_{2r}}e^{2u}|f_t|^{2} \varphi^2 (t_0) + 8 \int_{t_0}^{t} \int_{B_{3r/2}}|f_t|^{2} |\nabla \varphi|^2 \right)\\
&\le e^{2a(t-t_0)} \left( \int_{B_{2r}}e^{2u}|f_t|^{2} \varphi^2 (t_0) + \frac{8(4)^2}{r^2} \int_{t_0}^{t} \int_{B_{2r}}|f_t|^{2} \psi^2 \right)\\
&\le e^{2a(t-t_0)}\int_{B_{2r}}e^{2u}|f_t|^{2} \varphi^2 (t_0)\\
&\quad  + e^{2a(t-t_0)} \frac{8(4)^2}{r^2}e^{2at} 4^2 \varepsilon_1 \left( 1 + \frac{e^{2at}-e^{2at_0}}{2ar^2} \right).
\end{align*}
Now take $t_0 \in [t-\delta r^2,t]$ such that
\begin{equation*}
\int_{B_{2r}}e^{2u}|f_t|^2 \varphi^2 (t_0) = \min_{t-\delta r^2 \le s \le t} \int_{B_{2r}}e^{2u}|f_t|^2 \varphi^2 (s).
\end{equation*}
Then by \eqref{e^2u f_t^2 estimate},
\begin{equation*}
\int_{B_{2r}}e^{2u}|f_t|^{2} \varphi^2 (t_0) \le \frac{1}{\delta r^2}\int_{t-\delta r^2}^{t} \int_{B_{2r}}e^{2u}|f_t|^2 \varphi^2 \le \frac{1}{\delta r^2} 4^2 \varepsilon_1 \left( 1 + \frac{e^{2at}-e^{2a(t-\delta r^2)}}{2ar^2} \right).
\end{equation*}
Therefore,
\begin{align*}
\int_{B_{2r}} e^{2u}|f_t|^{2} \varphi^2(t) &\le 4^2 \varepsilon_1 \left( 1 + \frac{e^{2at}-e^{2at-2a\delta r^2}}{2ar^2} \right) \left( \frac{1}{\delta r^2} + \frac{8 (4)^2}{r^2} e^{2at} \right) e^{2a\delta r^2}.
\end{align*}
This completes the proof.
\end{proof}

\begin{cor}
Under the same assumption as above, we also have
\begin{align}
\int_{t-\delta r^2}^{t} \int_{B_{2r}} |\nabla f_t|^2 \varphi^2 &\le C C_1(r,\delta,t) C_2(r,\delta,t) \varepsilon_1 \label{f_ti^2 estimate}\\
\int_{t-\delta r^2}^{t} \int_{B_{2r}} |df|^2 |f_t|^2 \varphi^2 &\le C C_1(r,\delta,t) C_2(r,\delta,t)\varepsilon_1. \label{df^2 f_t^2 estimate}
\end{align}
\end{cor}

\begin{proof}
From the equation \eqref{p+2 derivative} with $p=0$, we can integrate from $t-\delta r^2$ to $t$.
\begin{align*}
\left. \int_{B_{2r}} e^{2u}|f_t|^{2} \varphi^2 \right|^t_{t-\delta r^2} &\le 2a \int_{t-\delta r^2}^{t}\int_{B_{2r}}e^{2u}|f_t|^{2}\varphi^2 + 8 \int_{t-\delta r^2}^{t} \int_{B_{2r}}|f_t|^{2}|\nabla \varphi|^2\\
&- \frac{1}{2}\int_{t-\delta r^2}^{t} \int_{B_{2r}}|\nabla f_t|^{2} \varphi^2\\
&+ \left( 2C_N + C_N^2 - 2b \right)\int_{t-\delta r^2}^{t} \int_{B_{2r}}|df|^2 |f_t|^{2}\varphi^2.
\end{align*}
Hence, we have
\begin{align*}
\frac{1}{2}\int_{t-\delta r^2}^{t} \int_{B_{2r}}|\nabla f_t|^2 \varphi^2 &\le 2C_1(r,\delta,t) C_2(r,\delta,t) \varepsilon_1 + 2a C_1(r,\delta,t) \varepsilon_1 + 8 \frac{4^2}{r^2} e^{2at} C_1(r,\delta,t) \varepsilon_1\\
&\le C C_1(r,\delta,t) C_2(r,\delta,t)\varepsilon_1.
\end{align*}
The other inequality is similar.
\end{proof}

\subsection{Higher estimate for time derivatives}

In this subsection we will get estimate for $e^{2u}|f_t|^{4}$.
We first build up $(p+2)$-version of \Cref{2 estimate}.

\begin{prop}
For solutions $(f,u)$ of \eqref{main0} and for $p \ge 1$, we have
\begin{equation} \label{p+2 estimate}
\begin{aligned}
\int_{t_1}^{t_2} \int_{B_{2r}}e^{2u}|f_t|^{p+2}\varphi^2 &\le C \int_{t_1}^{t_2} \int_{B_{2r}} |f_{ti}|^2 |f_t|^{p-1} \varphi^2 + C \int_{t_1}^{t_2} \int_{B_{2r}}|f_t|^{p+1}|\nabla \varphi|^2\\
&\quad + C \int_{t_1}^{t_2}\int_{B_{2r}}|df|^2 |f_t|^{p+1}\varphi^2.
\end{aligned}
\end{equation}
\end{prop}

\begin{proof}
First note that for any $p\ge 1$, $\nabla_i |f_t|^p = p |f_t|^{p-2}\langle f_{ti},f_t \rangle$.
Also, for simplicity, denote $\int \int = \int_{t_1}^{t_2} \int_{B_{2r}}$.
Multiplying $\tau(f)$ to \eqref{main0-1} gives
\begin{equation*}
2e^{2u}|f_t|^2 = -2 \langle f_{ti},f_i \rangle + \nabla_i ( 2 \langle f_t,f_i \rangle ).
\end{equation*}
Multiplying $|f_t|^p \varphi^2$ for $p \ge 1$ and integrating gives
\begin{align*}
2 \int \int e^{2u}|f_t|^{p+2}\varphi^2 &= -2 \int \int\langle f_{ti},f_i \rangle |f_t|^p \varphi^2 - 4 \int \int \langle f_t,f_i \rangle |f_t|^p \varphi \nabla_i \varphi\\
&\qquad - 2p \int \int \langle f_t,f_i \rangle \varphi^2 |f_t|^{p-2} \langle f_{ti},f_t \rangle\\
&= I + II + III.
\end{align*}
Now
\begin{align*}
I &\le C\int \int |f_{ti}|^2 |f_t|^{p-1} \varphi^2 + C\int \int|df|^2 |f_t|^{p+1} \varphi^2\\
II &\le C \int \int |f_t|^{p+1} |\nabla \varphi|^2 + C \int \int |df|^2 |f_t|^{p+1} \varphi^2\\
III &\le C \int \int |f_{ti}|^2 |f_t|^{p-1} \varphi^2 + C \int \int|df|^2 |f_t|^{p+1} \varphi^2.
\end{align*}
This completes the proof.
\end{proof}

Now we will show the desired estimate.

\begin{prop} \label{tau estimate}
Let $(f,u)$ are solutions of \eqref{main0}.
Assume that
\begin{equation*}
\sup_{T-2\delta r^2 \le t \le T}E(B_{2r},t) < \varepsilon_1.
\end{equation*}
Then for $t \in [T-\delta r^2,T]$, we have
\begin{equation} \label{e^2u f_t^4 estimate t}
\int_{B_{2r}} e^{2u}|f_t|^4 \varphi^2(t) \le C_3
\end{equation}
where
\begin{equation} \label{C_3 general}
C_3 = C C_1(r,\delta,t) C_2(r,\delta,t)^3 \varepsilon_1.
\end{equation}
\end{prop}

Note that $C_3$ depends on $r,t,\delta$.

\begin{proof}
For simplicity, denote $C_1 = C_1(r,\delta,t)$, $C_2 = C_2(r,\delta,t)$.
Also, denote $C$ for any number appeared in computations.
Suppose $\varphi$ be a cut-off function supported on $B_{3r/2}$ and $\varphi \equiv 1$ on $B_{r}$ and $|\nabla \varphi| \le \frac{4}{r}$.
Also, let $\psi$ be a cut-off function supported on $B_{2r}$ and $\psi \equiv 1$ on $B_{3r/2}$ and $|\nabla \psi| \le \frac{4}{r}$.
Let $t_1=t-\delta r^2$ and $t_2=t$.

The proof consists of several steps, increasing power of $|f_t|$.

{\bf Step 1.} Estimate for $\int \int e^{2u}|f_t|^3 \varphi^2$.

From \eqref{p+2 estimate} with $p=1$ and using \eqref{f_t^2 estimate}, \eqref{f_ti^2 estimate} and \eqref{df^2 f_t^2 estimate}, we have
\begin{equation} \label{e^2u f_t^3 estimate}
\int_{t-\delta r^2}^{t} \int_{B_{2r}} e^{2u}|f_t|^3 \varphi^2 \le CC_1C_2 \varepsilon_1
\end{equation}
and
\begin{equation} \label{f_t^3 estimate}
\int_{t-\delta r^2}^{t} \int_{B_{2r}} |f_t|^3 \varphi^2 \le  e^{2at} CC_1C_2 \varepsilon_1.
\end{equation}

{\bf Step 2.} Estimate for $\int e^{2u}|f_t|^3 \varphi^2$.

Now let $t_0 \in [t-\delta r^2,t]$ be such that
\begin{equation*}
\int_{B_{2r}}e^{2u}|f_t|^3 \varphi^2 (t_0) = \min_{t-\delta r^2 \le s \le t} \int_{B_{2r}}e^{2u}|f_t|^3 \varphi^2 (s).
\end{equation*}
From \eqref{p+2 derivative result} with $p=1$ and using \eqref{e^2u f_t^3 estimate} and \eqref{f_t^3 estimate}, we have
\begin{align*}
\int_{B_{2r}}e^{2u}|f_t|^3 \varphi^2(t) &\le e^{4a\delta r^2} \left( \int_{B_{2r}}e^{2u}|f_t|^3 \varphi^2(t_0) + 12 \int_{t_0}^{t} \int_{B_{3r/2}}|f_t|^3 |\nabla \varphi|^2 \right)\\
&\le e^{4a\delta r^2} \left( \frac{1}{\delta r^2} \int_{t-\delta r^2}^{t} \int_{B_{2r}} e^{2u}|f_t|^3 \varphi^2 + 12 \frac{4^2}{r^2} \int_{t-\delta r^2}^{t}\int_{B_{2r}} |f_t|^3 \psi^2 \right)\\
&\le e^{4a\delta r^2} \left( \frac{1}{\delta r^2}C C_1C_2 \varepsilon_1 + 12 \frac{4^2}{r^2}  e^{2at} C C_1C_2 \varepsilon_1 \right)\\
&=  C C_1 C_2 \varepsilon_1 \left(\frac{1}{\delta r^2} + \frac{12(4)^2}{r^2} e^{2at}\right)  e^{4a\delta r^2}.
\end{align*}
So, simply,
\begin{equation} \label{e^2u f_t^3 estimate t}
\int_{B_{2r}}e^{2u}|f_t|^3 \varphi^2(t) \le C C_1 C_2^2 \varepsilon_1.
\end{equation}

{\bf Step 3.} Estimate for $\int \int |\nabla f_t|^2 |f_t| \varphi^2$ and $\int \int |df|^2 |f_t|^3 \varphi^2$.

From \eqref{p+2 derivative} with $p=1$, we can integrate from $t-\delta r^2$ to $t$.
\begin{align*}
\left. \int_{B_{2r}} e^{2u}|f_t|^{3} \varphi^2 \right|^t_{t-\delta r^2} &\le 4a \int_{t-\delta r^2}^{t}\int_{B_{2r}}e^{2u}|f_t|^{3}\varphi^2 + 12 \int_{t-\delta r^2}^{t} \int_{B_{2r}}|f_t|^{3}|\nabla \varphi|^2\\
&- \frac{3}{4}\int_{t-\delta r^2}^{t} \int_{B_{2r}}|\nabla f_t|^{2} |f_t| \varphi^2\\
&+ \left( 3C_N + \frac{3 C_N^2}{2} - 4b \right)\int_{t-\delta r^2}^{t} \int_{B_{2r}}|df|^2 |f_t|^{3}\varphi^2.
\end{align*}
Note that $3C_N + \frac{3 C_N^2}{2} - 4b < 0$.
Now, from \eqref{e^2u f_t^3 estimate}, \eqref{f_t^3 estimate}, and \eqref{e^2u f_t^3 estimate t}, we have
\begin{align*}
\frac{3}{4}\int_{t-\delta r^2}^{t} \int_{B_{2r}}|\nabla f_t|^2 |f_t| \varphi^2 &\le 2C C_1C_2^2 \varepsilon_1 + 4a C C_1 C_2 \varepsilon_1 + 12 \frac{4^2}{r^2} e^{2at} CC_1 C_2 \varepsilon_1\\
&\le C C_1 C_2^2 \varepsilon_1.
\end{align*}
So, we have
\begin{equation} \label{f_ti^2 f_t estimate}
\int_{t-\delta r^2}^{t} \int_{B_{2r}}|\nabla f_t|^2 |f_t| \varphi^2 \le C C_1 C_2^2 \varepsilon_1.
\end{equation}
Similarly,
\begin{equation} \label{df^2 f_t^3 estimate}
\int_{t-\delta r^2}^{t} \int_{B_{2r}}|df|^2 |f_t|^{3}\varphi^2 \le C C_1 C_2^2 \varepsilon_1.
\end{equation}

{\bf Step 4.} Estimate for $\int \int e^{2u}|f_t|^4 \varphi^2$.

From \eqref{p+2 estimate} with $p=2$ and using \eqref{f_t^3 estimate}, \eqref{f_ti^2 f_t estimate} and \eqref{df^2 f_t^3 estimate}, we have
\begin{equation} \label{e^2u f_t^4 estimate}
\int_{t-\delta r^2}^{t} \int_{B_{2r}} e^{2u}|f_t|^4 \varphi^2 \le CC_1C_2^2 \varepsilon_1
\end{equation}
and
\begin{equation} \label{f_t^4 estimate}
\int_{t-\delta r^2}^{t} \int_{B_{2r}} |f_t|^4 \varphi^2 \le e^{2at}C C_1C_2^2 \varepsilon_1.
\end{equation}

{\bf Step 5.} Estimate for $\int e^{2u}|f_t|^4 \varphi^2$.

Now let $t_0 \in [t-\delta r^2,t]$ be such that
\begin{equation*}
\int_{B_{2r}}e^{2u}|f_t|^4 \varphi^2 (t_0) = \min_{t-\delta r^2 \le s \le t} \int_{B_{2r}}e^{2u}|f_t|^4 \varphi^2 (s).
\end{equation*}
From \eqref{p+2 derivative result} with $p=2$ and using \eqref{e^2u f_t^4 estimate} and \eqref{f_t^4 estimate}, we have
\begin{align*}
\int_{B_{2r}}e^{2u}|f_t|^4 \varphi^2(t) &\le e^{6a \delta r^2} \left( \int_{B_{2r}}e^{2u}|f_t|^4 \varphi^2(t_0) + 16 \int_{t_0}^{t} \int_{B_{3r/2}}|f_t|^4 |\nabla \varphi|^2 \right)\\
&\le e^{6a \delta r^2} \left( \frac{1}{\delta r^2} \int_{t-\delta r^2}^{t} \int_{B_{2r}} e^{2u}|f_t|^4 \varphi^2 + 16 \frac{4^2}{r^2} \int_{t-\delta r^2}^{t}\int_{B_{2r}} |f_t|^3 \psi^2 \right) \\
&\le e^{6a \delta r^2} \left( \frac{1}{\delta r^2} CC_1C_2^2 \varepsilon_1 + 16 \frac{4^2}{r^2}  e^{2at} CC_1C_2^2 \varepsilon_1 \right) \\
&=  C C_1 C_2^2 \varepsilon_1 \left(\frac{1}{\delta r^2} + \frac{16(4)^2}{r^2} e^{2at}\right) e^{6a \delta r^2}.
\end{align*}
So, simply,
\begin{equation*}
\int_{B_{2r}}e^{2u}|f_t|^4 \varphi^2(t) \le C C_1 C_2^3 \varepsilon_1.
\end{equation*}
\end{proof}

\begin{remark}
We can keep going on to get bounds for $\int_{B_{2r}}e^{2u}|f_t|^n \varphi^2(t) \le C_3(n)$ for any $n$.
However, these bounds blow up to infinity as $n \rightarrow \infty$.
\end{remark}

\section{$W^{2,2}$ and gradient estimate}
\label{sec7}

In this section we will get $W^{2,2}$ estimate and gradient estimate for the solution $f$ of \eqref{main0-1}.
For simplicity, denote $\| \cdot \|_{k,p} = \| \cdot \|_{W^{k,p}(B_{2r})}$ and $\| \cdot \|_p = \| \cdot \|_{0,p}$.
First observe the following.

\begin{lemma}
Let $u$ be a solution of \eqref{main0-2}.
For $p>2$ and for any $r>0$,
\begin{equation} \label{e^pu estimate}
\int_{B_{2r}}e^{pu}\varphi^r(t) \le \int_{B_{2r}}e^{pu}\varphi^r(t_0) + \frac{2b^2 (p-2)}{pa} \int_{t_0}^{t} \int_{B_{2r}} |df|^p\varphi^r.
\end{equation}
\end{lemma}

\begin{proof}
Note that
\begin{equation*}
\partial_t (e^{pu}) = p e^{pu}u_t = pb e^{(p-2)u}|df|^2 - pae^{pu}.
\end{equation*}
So, multiplying $\varphi^r$ and integrating over $B_{2r}$ gives
\begin{align*}
\frac{d}{dt}\int_{B_{2r}} e^{pu} \varphi^r &= pb \int_{B_{2r}}e^{(p-2)u}|df|^2\varphi^r - pa \int_{B_{2r}} e^{pu}\varphi^r\\
& \le b \lambda (p-2) \int_{B_{2r}}e^{pu}\varphi^r + 2b \lambda^{-1} \int_{B_{2r}}|df|^p\varphi^r - pa \int_{B_{2r}}e^{pu}\varphi^r\\
&= \frac{2b^2 (p-2)}{pa} \int_{B_{2r}}|df|^p\varphi^r
\end{align*}
by Young's inequality with weight $\lambda = \frac{pa}{b(p-2)}$.
Hence, by integrating, we obtain the result.
\end{proof}

\begin{lemma}\label{df^rq estimate}
Let $f$ be any smooth function and let $\varphi \in C^{\infty}_{0}(B_{2r})$ be a cut-off function.
Then for any $r>1$ and $p \ge 2$, we have
\begin{equation}
\| \,|df|^r \varphi \|_{p} \le C \| \,|df|^{r-1}\|_{p} \| f \varphi\|_{2,2}.
\end{equation}
\end{lemma}

\begin{proof}
Let $1 \le s < 2$ be such that $p = 2s(2-s)$.
By Sobolev embedding,
\begin{align*}
\| \,|df|^r \varphi \|_{p} &\le C \| \nabla (|df|^r \varphi) \|_{s}\\
&\le C \| |df|^{r-1} \nabla (|df| \varphi) \|_{s}\\
&\le C \| \,|df|^{r-1}\|_{p} \| f \varphi\|_{2,2}.
\end{align*}
\end{proof}

Next, we will show $W^{2,2}$ estimate.

\begin{prop} \label{W^22 estimate}
Let $(f,u)$ are solutions of \eqref{main0}.
Then there exists $\varepsilon_1>0$ such that the following holds:

Assume that
\begin{equation*}
\sup_{T-2\delta r^2 \le t \le T}E(B_{2r},t) \le \varepsilon_1 , \qquad \int_{B_{2r} \times \{T-2\delta r^2\}}e^{6u} \le \varepsilon_1.
\end{equation*}
Then for $t \in [T-\delta r^2,T]$, we have
\begin{equation}
\|f\varphi\|_{2,2} \le C_4  = C_4 (r,\delta,T,\varepsilon_1,C_N)
\end{equation}
where
\begin{equation*}
C_4  = \left(C C_3 \varepsilon_1 + C \varepsilon_1^4\right)^{1/4} \left( 1 + C_3 \varepsilon_1 \delta r^2 \exp(C_3 \varepsilon_1 \delta r^2 ) \right)^{1/4}.
\end{equation*}
\end{prop}

\begin{proof}
Suppose $\varphi$ be a cut-off function supported on $B_{3r/2}$ and $\varphi \equiv 1$ on $B_{r}$ and $|\nabla \varphi| \le \frac{4}{r}$.
Also, let $\psi$ be a cut-off function supported on $B_{2r}$ and $\psi \equiv 1$ on $B_{3r/2}$ and $|\nabla \psi| \le \frac{4}{r}$.
Let $t_0 = T-2\delta r^2$.

Without loss of generality, assume $\int_{\Omega} f = 0$.
Then we have, by Poincare,
\begin{equation*}
\|f\|_{p} \le C_p \|df\|_{p}.
\end{equation*}
From the equation $\Delta f + A(df,df) = e^{2u}f_t$, multiplying $\varphi$ and arranging terms gives
\begin{align*}
|\Delta (f\varphi)| &\le |A(df,df) \varphi| + |e^{2u}f_t \varphi| + k(\varphi) (|f| + |df|)\\
&\le C_N | \, |df|^2 \varphi| + |e^{2u}f_t \varphi| + k(\varphi) (|f| + |df|).
\end{align*}

By the $L^p$ estimate, we have
\begin{equation} \label{L^p estimate}
\|f\varphi\|_{2,p} \le C \left(C_N \|\, |df|^2 \varphi \|_{p} + \|e^{2u} |f_t| \varphi \|_{p} + \|df\|_{p} \right)
\end{equation}
where the constant $C$ only depends on $p$ and $r$.

Now let $p=2$.
Note that, by \eqref{e^2u f_t^4 estimate t} and \eqref{e^pu estimate},
\begin{align*}
\|e^{2u} |f_t| \varphi \|_{2}^4 & = \left(\int_{B_{2r}}e^{4u}|f_t|^2 \varphi^2 \right)^2\\
&\le \left(\int_{B_{3r/2}}e^{2u}|f_t|^4 \right) \left(\int_{B_{2r}} e^{6u}\varphi^4\right)\\
&\le \left(\int_{B_{2r}}e^{2u}|f_t|^4 \psi^2\right) \left(\int_{B_{2r}} e^{6u}\varphi^4\right)\\
&\le C_3 \left(\int_{B_{2r}} e^{6u}\varphi^4 (t_0) + \frac{8b^2}{6a} \int_{t_0}^{t} \int_{B_{2r}}|df|^6 \varphi^4\right)\\
&\le C_3 \varepsilon_1 + C_3 \int_{t_0}^{t}\int_{B_{2r}}|df|^6 \varphi^4.
\end{align*}

Now applying \Cref{df^rq estimate} with $r=3/2$,$q=4$ gives
\begin{align*}
\left(\int_{B_{2r}}|df|^6 \varphi^4 \right)^{1/4} &= \| \,|df|^{3/2} \varphi \|_{4} \le C \| \,|df|^{1/2}\|_{4} \| f \varphi\|_{2,2}\\
& \le C \varepsilon_1^{1/4} \| f \varphi\|_{2,2}.
\end{align*}

On the other hand, applying \Cref{df^rq estimate} with $r=2$, $q=2$ gives
\begin{equation*}
\| \,|df|^2 \varphi \|_{2} \le C \| df \|_{2} \| f \varphi\|_{2,2} \le C \varepsilon_1^{1/2} \| f \varphi\|_{2,2}.
\end{equation*}

All together, we have
\begin{equation*}
\|f\varphi\|_{2,2}^4 \le C C_N^4 \varepsilon_1^2 \| f \varphi\|_{2,2}^4 + C C_3 \varepsilon_1 + C C_3 \varepsilon_1 \int_{t_0}^{t} \| f \varphi\|_{2,2} + C \varepsilon_1^4
\end{equation*}

Let $X = \|f\varphi\|_{2,2}^4$.
Then the above equation becomes
\begin{equation*}
(1- C C_N^4 \varepsilon_1^2) X \le C C_3 \varepsilon_1 + C \varepsilon_1^4 +C_3 \varepsilon_1 \int_{t_0}^{t} X.
\end{equation*}

So, if $\varepsilon_1$ is small enough so that $1- C C_N^4 \varepsilon_1^2 >1/2$, then by Gronwall's inequality, we have
\begin{align*}
\|f\varphi\|_{2,2}^4 &\le \left(C C_3 \varepsilon_1 + C \varepsilon_1^4\right) \left( 1 + C_3 \varepsilon_1 (t-t_0) \exp(C_3 \varepsilon_1 (t-t_0)) \right)\\
&\le  \left(C C_3 \varepsilon_1 + C \varepsilon_1^4\right) \left( 1 + C_3 \varepsilon_1 \delta r^2 \exp(C_3 \varepsilon_1 \delta r^2) \right).
\end{align*}
This completes the proof.
\end{proof}

From Sobolev embedding, we now have, for $t \in [T-\delta r^2,T]$,
\begin{equation}\label{W^1p estimate}
\|f\varphi\|_{1,p} \le C_4 
\end{equation}
for any $p > 1$.

Now we will show gradient estimate.
This can be achieved by obtaining better estimate than $W^{2,2}$, say $W^{2,3}$.

\begin{prop} \label{W^23 estimate}
Assume the same as in \Cref{W^22 estimate}.
In addition, we assume that
\begin{equation*}
\int_{B_{2r} \times \{T-2\delta r^2\}}e^{18u} \le \varepsilon_1.
\end{equation*}

Then for $t \in [T-\delta r^2,T]$, we have
\begin{equation}
\|f\varphi\|_{2,3} \le C_5  = C_5 (r,\delta,T,\varepsilon_1,C_N)
\end{equation}
where
\begin{equation*}
C_5  = C \left( C_NC_4 ^2 + C_3^{1/12} \varepsilon_1^{1/12} + C_3^{1/12} \delta^{1/12} r^{1/6} C_4 ^{3/2} +  C_4  \right).
\end{equation*}

In particular,
\begin{equation}
\sup_{B_r} |df| \le C_5 .
\end{equation}
\end{prop}

\begin{proof}
By \Cref{W^1p estimate}, we have uniform bound for $|df|^p$ for any $p$.
Now from equation \eqref{L^p estimate}, we have
\begin{equation*}
\|f\varphi\|_{2,p} \le C \left(C_N \|\, df \|_{2p}^2 + \|e^{2u} |f_t| \varphi \|_{p} + \|df\|_{p} \right).
\end{equation*}

Now let $p=3$ and $t_0 = T-2\delta r^2$.
Then we have, using \eqref{e^pu estimate} and \Cref{W^1p estimate},
\begin{align*}
\|e^{2u} |f_t| \varphi \|_{3}^{12} & \le \| e^{u/2}|f_t| \varphi\|_{4}^{12} \|e^{3u/2}\|_{12}^{12}\\
&\le C_3 \int_{B_{2r}}e^{18u}\\
&\le C_3 \left(\int_{B_{2r}} e^{18u} (t_0) + C \int_{t_0}^{t} \int_{B_{2r}}|df|^{18} \right)\\
&\le C_3 \varepsilon_1 + C C_3 \delta r^2 C_4 ^{18}.
\end{align*}

Applying \Cref{W^1p estimate} completes the proof.
\end{proof}

\section{Global weak solution}
\label{sec8}

In this section, we will prove the main theorem \ref{global solution}.

\begin{lemma}
There exists $\ep_1>0$ such that if $(f,u)$ be a smooth solution of \eqref{main0} on $B_{2r} \times [T-2\delta r^2,T]$ and
\begin{equation}
\sup_{T-2\delta r^2 \le t \le T}E(B_{2r},t) \le \ep_1  \quad \textrm{ and } \quad  \int_{B_{2r} \times \{T-2\delta r^2\}}e^{18u} \le \ep_1, \label{ep1 cond}
\end{equation}
then H{\"o}lder norms of $f,u$ and their derivatives are all bounded by constants only depending on $T,r,\delta,\ep_1,C_N$.
\end{lemma}

\begin{proof}
By the sup bound of $|df|$, we have $e^{-2u(f)} \le e^{2aT}$ and
\begin{align*}
e^{-2u(f)} &= \frac{e^{2at}}{1+2b \int_{0}^{t}e^{2as}|df|^2(x,s)ds} \ge \frac{e^{2at}}{1+2b {M'} ^2 \frac{e^{2at}-1}{2a}}\\
& \ge \frac{1}{1 + \frac{b}{a} {M'} ^2}.
\end{align*}
Hence the operator $\partial_t - e^{-2u}\Delta$ is uniformly parabolic on $[0,T_0)$.

Similar in proof of Theorem \ref{short time}, we conclude the desired estimate.
\end{proof}

\begin{proof}
(Proof of Theorem \ref{global solution})
First consider $f_0$ is smooth.
By Theorem \ref{short time}, there exists a smooth solution in $(\Sigma \times [0,T))$ for some $T>0$.
Let $T_1$ be the maximal existence time.
If $T_1 = \infty$ then we obtain global solution which is smooth everywhere.
So suppose $T_1 < \infty$.

If we have $\limsup_{t \nearrow T_1} E(B_{2r}(x),t) \le \ep_1$ for any $x \in \Sigma$ and $r>0$, then by above lemma H{\"o}lder norms of $f,u$ and their derivatives are all bounded, hence $f,u$ can be extended beyond the time $T_1$.
This contradicts with maximality of $T_1$.
So there should be a point $x \in \Sigma$ such that
\[
\limsup_{t \nearrow T_1} E(B_{2r}(x),t) > \ep_1.
\]
Since the total energy is finite, there are at most finitely many such points $\{x_1, \cdots, x_{k_1}\}$.
Then by above lemma,
we get smooth solution $(f_1,u_1)$ on $\Sigma \times [0,T] \setminus \{(x_i^1,T_1)\}_{i=1, \cdots, k_1}$.
If we denote $f(x,T_1)$ and $u(x,T_1)$ as the weak limit of $f(x,t)$ and $u(x,t)$ as $t \nearrow T_1$, then $f(t),u(t)$ converges to $f(T_1),u(T_1)$ strongly in $W^{1,2}_{loc}(\Sigma \setminus \{x_i^1\})$.

Next, denote $g_1 = e^{2u_1(x,T_1)}g_0$ and consider the flow \eqref{main0} with initial map $f_1$ and initial metric $g_1$.
As above, there is a smooth solution $(f_2,u_2)$ on $\Sigma \times [0,T_2] \setminus \{(x_i^2,T_2)\}_{i=1, \cdots, k_2}$.
From these we can set up a smooth solution $(f,u)$ on $\Sigma \times [0,T_1+T_2]$ which is smooth except $\{(x_i^1,T_1)\} \cup \{(x_i^2,T_2)\}$.
Iterate this process to obtain global solution with exception points, which are at most finitely many because the total energy is finite.

\end{proof}

\section{Finite time singularity}
\label{sec9}

As the conformal heat flow is developed to postpone the finite time singularity, it is expected to have no finite time singularity.
In this section we will discuss few remarks about finite time singularity.

Recall the following
\begin{lemma}
(\cite {T04b})
There exist a compact target manifold $N$, a smooth map $f_0 : D \to N$ and $\ep>0$ such that every smooth map $f : D \to N$ homotopic to $v_0$ fails to be harmonic.
If furthermore $E(f) \le E(f_0)$, then
\[
\int_{D} |\tau(f)|^2 \ge \ep.
\]
\end{lemma}
Together with energy decreasing property of harmonic map heat flow $f(t)$, the above lemma implies that no heat flow starting with initial map $f_1$ homotopic to $f_0$ above can be smooth after the time $t = \frac{E(f_1)}{\ep}$.

This argument can be avoided in conformal heat flow.
From \eqref{E'<0}, we have
\[
E(0)-E(t) = \int_{0}^{t} \int_{D} e^{-2u}|\tau(f(t))|^2.
\]
So, if $u$ is large, $\int_{D} e^{-2u}|\tau(f(t))|^2$ can be smaller than $\ep$ even if $\int_{D} |\tau(f(t))|^2 > \ep$.

The proof of the above lemma relies on no-neck property of approximate harmonic map with $\|\tau\|_{L^2} \to 0$.
And the assumption $\|\tau\|_{L^2} \to 0$ is essential in the no-neck property as there is a counter example of Parker without the assumption.
The conformal heat flow makes the tension field converge to zero with different scale.
Hence the information about the converging scale of the tension field will play an important role in the property of the flow.

\section*{Acknowledgement}

The author would like to thanks Armin Schikorra and Thomas Parker for valuable comments and advice.


\bibliographystyle{abbrv}
\bibliography{bib}

\end{document}